\documentclass[11pt]{article}

\usepackage{amsmath,amssymb,amsthm,textcomp}
\usepackage{bm}
\usepackage{enumerate}
\usepackage{multicol}
\usepackage{graphicx}
\usepackage{wrapfig}
\usepackage{pinlabel}
\usepackage{abstract}

\graphicspath{ {./Images/ } }
\usepackage[colorlinks=true, allcolors=blue]{hyperref}
\usepackage{fullpage}

\theoremstyle{plain}
\newtheorem{theorem}{Theorem}
\newtheorem{lemma}[theorem]{Lemma}

\newtheorem{corollary}[theorem]{Corollary}

\theoremstyle{plain} 
\newcommand{\thistheoremname}{}
\newtheorem*{genericthm*}{\thistheoremname}
\newenvironment{namedthm*}[1]
  {\renewcommand{\thistheoremname}{#1}%
   \begin{genericthm*}}
  {\end{genericthm*}}

\theoremstyle{definition}
\newtheorem{definition}[theorem]{Definition}

\newcommand{\lra}{\longrightarrow}

\newcommand{\Lra}{\Longrightarrow}
\newcommand{\mbf}{\mathbf}

\newcommand{\mbb}{\mathbb}

\makeatletter
\renewcommand{\maketitle}{
\begin{center}
\vspace{1ex}
{\huge \textsc{\@title}}\\
\vspace{2ex}
{\large\textsc{\@author}}
\end{center}
}
\makeatother

\author{Deepisha Solanki}

\begin{document}

\title{Studying links via plats: the unlink}

\maketitle


\begin{abstract}

\textbf{Abstract.} Our main result is a version of Birman's theorem about equivalence of plats, which does not involve stabilization, for the unlink. We introduce the {\em pocket} and {\em flip moves}, which modify a plat without changing its link type or bridge index. Theorem 1 shows that using the {\em pocket} and {\em flip moves}, one can simplify any closed $n$-plat presentation of the unknot to the standard 0-crossing diagram of the unknot, through a sequence of plats of \textbf{non-increasing} bridge index. The theorem readily generalises to the case of the unlink.

\end{abstract}

\section{\centering Introduction}

In this paper we focus on plat presentations of non-oriented links in $\mathbb{R}^3$.  In particular, we give a procedure for the {\em monotonic} simplification, with respect to bridge index, of any $k$-bridge plat presentations of the $n$-component unlink (with $k \geq n$) to the standard (i.e. trivial $2n$ braid) $n$-bridge plat presentation. This procedure introduces two new bridge index preserving plat isotopies: the {\em flip move} and the {\em pocket move}. Both of these isotopies/moves have the attributes that they take plats to plats and that they preserve bridge index. (See \S \ref{dcms} and \S \ref{flipmove} for formal definitions). Our procedure couples these two isotopies with the classical ones of {\em braid isotopy}, stabilization/destabilization and {\em double coset moves} (these moves are defined in \S \ref{braid_isotopies}, \S \ref{stab_destab} and \S \ref{dcms} respectively) for our main result which is specialized to the unknot:

\begin{theorem} Let $\mbf{K}$ be any 2n-plat presentation of the unknot.  Let $\mathbf{U}$ be the standard 1-bridge, zero-crossing diagram of the unknot. Then, there exists a finite sequence of plat presentations of the unknot:
$$ \mathbf{K} = \mathbf{K_{0}} \lra \mathbf{K_{1}} \lra \mathbf{K_{2}} \lra ... \lra \mathbf{K_{m}} = \mathbf{U} $$
such that $\mathbf{K_{i+1}}$ is obtained from $\mathbf{K_{i}}$ via the following \textbf{non index increasing} moves:

\begin{itemize}
\item[(i)] a braid isotopy on the 2n strands;
\item[(ii)] destabilization;
\item[(iiI)] the pocket move;
\item[(iv)] the flip move.
\end{itemize}
\label{thm1}

\end{theorem}

Our methods for establishing this result are geometric and can readily be generalized to give us a similar statement for the unlink. See Corollary \ref{case_unlink} below:

\begin{corollary} 

\label{case_unlink}

Let $\mbf{L}$ be any {2n-plat} presentation of the unlink of $k$ components.  Let $\mathbf{U_k}$ be the standard $k$-bridge, zero-crossing diagram of the unlink. Then, there exists a finite sequence of plat representatives of the unlink:
$$ \mathbf{L} = \mathbf{L_{0}} \lra \mathbf{L_{1}} \lra \mathbf{L_{2}} \lra ... \lra \mathbf{L_{m}} = \mathbf{U_k} $$
such that $\mathbf{L_{i+1}}$ is obtained from $\mathbf{L_{i}}$ via the moves mentioned in Theorem \ref{thm1}.

\end{corollary}

This is a striking result because there exist examples of 2-bridge unknots (presented in \ref{goeritz}), which can not be simplified to the trivial unknot without first increasing the number of crossings. Further, we can construct examples of 2-bridge unknots which need {\em arbitrarily} many extra crossings before they can be simplified to the trivial diagram. But the flip move enables us to do the simplification process without adding any extra crossings.


\subsection{Philosophical Motivation}

Between 1990 and 2006, Birman and Menasco wrote a sequence of papers \cite{bm_I, bm_II, bm_III, bm_IV, bm_V, bm_VI, bm_mtwsi} with the running major title ``Stabilization in the braid groups’’, that investigated the question of what stabilization accomplishes in Markov's theorem. One key result coming out of that sequence of papers was the discovery of a new isotopy—the {\em exchange move}---which also preserves braid index and writhe, but could jump between conjugacy classes. Thus, in \cite{bm_V}, it was established that using just a sequence of braid isotopies, exchange moves and destabilizations one could go from any closed braid representative of the unlink to the representative of minimal index—the trivial closed braid. In \cite{bm_IV} it was shown that, using a sequence of braid isotopies and exchange moves, one could take a closed braid representing a split or composite link to one where it was ``obviously’’ split or composite, respectively.

In the setting of plat presentations of links, we now ask the same question: {\em what does stabilization accomplish?} This note is the first of a planned sequence of papers modeled on the Birman---Menasco sequence addressing this question. We introduce two new isotopy moves that preserve the bridge index: the {\em pocket move} (see \S \ref{dcms}) which rightly should be thought of as a geometrization of a double coset move (see \S \ref{dcms}); and the {\em flip move} (see \S \ref{flipmove}) which the reader might think of as being analogous to an exchange move of the closed braid setting.

\begin{figure}[ht!]
\labellist
\small\hair 2pt
\pinlabel {$n$-bridges} at 208 460
\pinlabel {$n$-bridges} at 205 22
\pinlabel {$\textbf{B}$} at 204 235
\pinlabel {$2n$-strands} at 450 298
\endlabellist
\centering
\includegraphics[width=9cm, height=7cm]{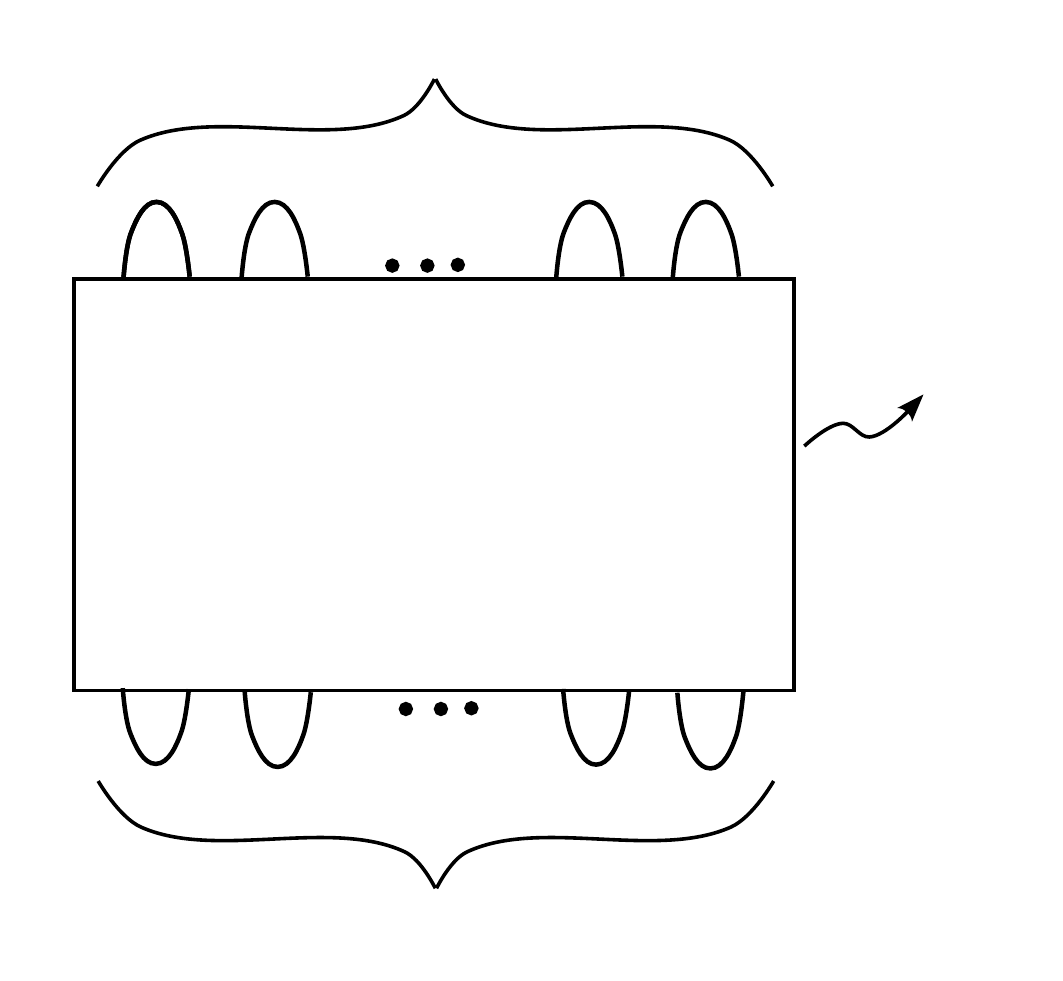}
\caption{A generic 2n-plat}
\label{plat}
\end{figure}
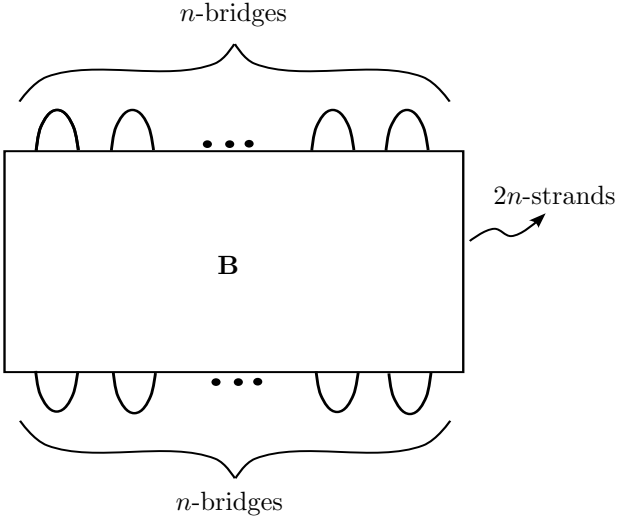


\subsection{Idea of proof}

Our strategy for establishing Theorem \ref{thm1} is inspired by the singular foliation technology coming from the Birman---Menasco papers. There they considered an appropriate essential surface in the complement of the closed braid and studied the singular foliation induced by the braid fibration.

In our setting we start with a pair $\mbf{(D,K)}$-spanning disc and plat presentations of the unknot. We then consider the singular foliation induced on $D$ coming from the plat height function. The  singular foliation technology for proving Theorem \ref{thm1} will have the following salient components: 

(i) A system of level curves and arcs corresponding to the regular values of the height function (See \S \ref{level curves}) \\
(ii) Generic singular leaves corresponding to the critical values of the height function. The foliated neighborhood of each singular leaf will form a {\em singular tile}.  These tiles will cover the spanning disc. (See \S \ref{tilingandfoliation}) \\
(iii) The covering of the spanning disc by singular tiles will yield an Euler characteristic equation (See Lemma \ref{surfaceec}), this equation is used to obtain a combinatorial relation between different types of tiles. (See Lemma \ref{lemma6}\\
(iv) The tiling of the spanning disc will be ``dual'' to a finite trivalent graph that allows us to define a complexity measure on the tiling. (See \S \ref{graphcomplexity}) \\
(v) The local geometric realisation of certain tiling patterns will imply the plat admits a pocket or flip move. Applying a pocket or flip move will strictly reduce the complexity measure. 

In \S \ref{tilingandfoliation}, we describe the `tiles' associated with the foliation, each tile type will contain a single singularity. 

This gives an obvious complexity measure on the pair $\mbf{(D,K)}$, counting the number of singularities of certain types. This is rigorously defined in \S \ref{graphcomplexity}. 

In \S \ref{tilingandfoliation}, we explain how the ``configuration" in the singular foliation corresponds to the plat moves and results in \S \ref{mainresult} detail how plat moves remove these singularities, thus reducing the complexity measure on the spanning disc of the unknot.


\subsection{Outline of the paper} 

In \S \ref{prelims}, we introduce the basic terminology and set up the required machinery. \S \ref{prelims} begins with a description of the system of level curves in \S \ref{level curves}. This is followed by an explanation of moves on plats in \S \ref{platmoves}. Then, we have a detailed description of the different types of singularities that occur on the spanning disc (with respect to the plat height function), and how that set up is used to foliate the disc and construct a tiling of it, in \S \ref{tilingandfoliation}. \S \ref{prelims} concludes with an introduction to the graph corresponding to a disc foliation and how that is used to define a complexity function on the disc, in \S \ref{graphcomplexity}. 

Moving on to the next section, \S \ref{mainresult} consists of four lemmas leading up to the proof of Theorem \ref{thm1}, followed by the proof of Corollary \ref{case_unlink}, which extends Theorem \ref{thm1} to the unlink. Lemma \ref{surfaceec} gives the relation between the Euler characteristics of a surface $S$ and subsurfaces $S_i$'s such that the $S_i$'s are glued together to give $S$. We use this formula to obtain a combinatorial relation between different types of tiles in Lemma \ref{bla}. Lemma \ref{bla} is crucial in helping us understand the constraints on how different singularities on the spanning disc can grow with respect to each other. It also helps us identify some key characteristics of the graph associated with the foliation of the spanning disc. This leads us to Lemma \ref{lemma5} which tells us how to start eliminating singularities from the disc, using the language of graph theory. In particular, in Lemma \ref{lemma5}, we give conditions that a vertex in the graph needs to satisfy so that the singularity it represents can be removed using one of the moves in Theorem \ref{thm1}, giving us a plat with a strictly smaller value of the complexity function. In Lemma \ref{lemma6}, we prove that a vertex satisfying conditions of Lemma \ref{lemma5} always exists, if there are any singularities of a certain type on the disc. 

The rest of \S \ref{mainresult} is dedicated to proving the statement of Theorem \ref{thm1} and extending it to the unlink. 

In \S \ref{S^3}, we explore the algebraic situation in $S^3$, and prove that in $S^3$, there is only double coset class corresponding to the unknot! 

Finally, we talk about the Goeritz unknot and two different methods of untangling it: with and without the flip move. The reader can skip this section if they like, as it has no bearing on the rest of the paper.

\section*{Acknowledgements}

We thank William Menasco for introducing the author to this problem and his constant guidance throughout this research project. We thank Joan Birman for taking a very active interest in this research, proof-reading this manuscript and providing important insights to the author, and Dan Margalit for proof-reading this manuscript and offering helpful suggestions. 
We would also like to thank Ethan Dlugie for pointing us to Kunio Murasugi and Bohdan I. Kurpita's book \cite{murasugi_kurpita}, and Seth Hovland for a very careful reading of this manuscript and his comments on how to improve it. In chapter 4 of their book, Murasugi and Kurpita deal with a special class of braids and the algebraic properties of these braids. It turns out that these braids are the same as our `flip move' (described in \S \ref{flipmove}). We independently discovered the `flip' move, but interestingly, were motivated by the same phenomenon as Murasugi and Kurpita.

\section{\centering Preliminaries}

\label{prelims}

Beginning with plat presentations for non-oriented links, we start with a brief description of what they are. Consider the classical braid group on $2n$ strands, $\mathcal{B}_{2n}$. Then, for an element of $\mathcal{B}_{2n}$ i.e. a braid $B$ on $2n$ strands, we cap-off the $i$ and $i+1$ 
strands, for odd $i$ between $1$ and $2n-1$, at the top (bottom) of the braid with $n$ northern (southern) circle hemispheres, refer Figure \ref{plat}. Call this the plat defined by $B$.  Then $n$ is the {\em bridge index} of the presentation. We thus have a height function on the presentation that has $n$ maxima (minima) points all at the same level 
above (below) the braid. (We will employ this height function extensively in the arguments of this note). One can jump/isotope from one plat presentation to another through braid isotopies, which again are Reidemeister type-II \& -III moves placed in this setting. Braid isotopies do not change the bridge index. An isotopy by a stabilization (destabilization) adds (removes) two strands in the braid along with a maxima and minima, which is the Reidemeister type-I move placed in the plat setting. These isotopies are shown in Figures \ref{braidisotopy_one}, \ref{braidisotopy_two} and \ref{stab}.

Birman showed in \cite{birman_1976} that any link type can be represented as a plat. It turns out that this representation is highly non unique and every link type has {\em infinitely many} distinct plat presentations. By ``distinct plat presentation'', we mean that the braids corresponding to these plat presentations correspond to different elements in the braid group. In the same paper, Birman proved the following result, which tells us how any two braids, whose plats represent the same knot type, are algebraically related to one another:

\begin{namedthm*}{Birman's Result}

Let $\mbf{L_i}$, i =1, 2, be tame knots. Choose elements $\phi_i \in \mathcal{B}_{2n_i}$ in such a way that the plat defined by $\phi_i$ corresponds to the same knot type as $\mbf{L_i}$. Then, $\mbf{L_1} \approx \mbf{L_2} $ if and only if there exists an integer $t \geq$ max$(n_1, n_2)$ such that, for each $n \geq t$, the elements 
$$\phi_{i}^{'} = \phi_i \sigma_{2n_i} \sigma_{2n_i +2} ... \sigma_{2n-2} \in \mathcal{B}_{2n}, i = 1, 2$$
are in the same double coset of $\mathcal{B}_{2n}$ modulo the subgroup $\mathcal{K}_{2n}$.

\end{namedthm*}

Birman's Result can be restated as follows: 

\begin{namedthm*}{\textbf{Stable equivalence of plat presentations}}

Let $K$ be a $2n$-plat presentation of $\mbf{K}$, and $K'$ be a $2n^{\prime}$-plat presentation of $\mbf{K}$. Then, there is a finite sequence of plat presentations of $\mbf{K}$: 
$$ K = K_{0} \longrightarrow K_{1} \longrightarrow K_{2} \longrightarrow ... \longrightarrow K_{m} = K' $$
such that $K_{i+1}$ is obtained from $K_{i}$ via the following moves:
\begin{itemize}
\item[(i)] braid isotopies;
\item[(ii)] double coset moves;
\item[(iii)] stabilization or destabilization.
\end{itemize}
\end{namedthm*}

Let $\mbf{K}$ denote the isotopy class of the unknot and $\mbf{D}$ denote the disc bounded by $\mbf{K}$ in $\mbb{R}^3$. Let $\mbf{K_i}$ be a plat presentation of the unknot, from the sequence in Theorem \ref{thm1}, and let $D_i$ be a geometric realisation of $\mbf{D}$ corresponding to $\mbf{K_i}$. We assign a height function $h$ to $D_i$ such that $h$ is a Morse function, regarded as the projection function from $D_i$ onto $\mbb{R}$ (think of this as the natural height function $h: \mbb{R}^3 \lra \mbb{R}$. We position $D_i$ in $\mbb{R}^3$ so that $h$ takes values on $D_i$ varying between $-0.25$ and $1.25$ such that, the points of maxima of the top bridges are at height $h = 1$ and the points of minima of the bottom bridges are at height $h=0$. With this setup, thinking of $\mbb{R}^3$ as $\mbb{R}^2 \times \mbb{R}$, then $\mbb{R}^2 \times \{*\}$ are the level planes for $h$, say $P_t = h^{-1}(t) =  \mbb{R}^2 \times \{t\}$, for $t \in [-0.25, 1.25]$.

\subsection{System of level curves}   \label{level curves}

Define $C_t^i := D_i \cap P_t$. Then, $C_t^i$ is the set of level curves of $D_i$ at level $t$, which is basically a cross section of the disc at height $t$. Note that, if $\mbf{K}_i$ is a plat of index $n$, then, for $ t \in (0,1)$, $C_t^i$ consists of exactly $n$ arcs and $m_t^i$ closed curves, say. Call these arcs $a_{i,t}^j$, for $j \in \{1,2,..., n \}$ and the closed curves $s_{i,t}^j$, for $j \in \{1,2, ..., m_t^i\}$. We position $D_i$ so that, for $t$ values outside the interval $(0,1)$, $C_t^i$ only consists of simple closed curves, say $m_t^i$ many of them (same notation as for $t \in (0,1)$), and a discrete set of points, corresponding to local maxima and minima on $D_i$. If, for any $i,j,t$, $s_{i,t}^j$ does not contain any arc or curve inside of it, we surger it off. This means that, for any $i$ and any $t \in [-0.25, 1.25]$, the only closed curves that remain on $D_i$ are the ones that contain an arc or curve. Further, notice that, for all but finitely many values of $t$, the arcs $a_{i,t}^j$ are all simple arcs and the closed curves $s_{i,t}^j$ are all simple closed curves. The values of $t$ for which we get intersection (which includes self-intersection) between arcs or closed curves are the values where the disc $D_i$ has a saddle. We position the disc so that no value of $t$ has more than one saddle.

\subsection{Moves on plats} \label{platmoves}

\subsubsection{Braid isotopies} \label{braid_isotopies}

Braid isotopies are moves on plats where the top and bottom bridges (i.e. the local max's and min's) are fixed. Algebraically, this is equivalent to changing the braid word corresponding to the plat using the relators in the braid group. So the end points of the strings are all fixed and with that constraint, any manipulation of strings is allowed as long as the resulting object is still a plat. This is shown in Figure~\ref{braidisotopy_one} and Figure~\ref{braidisotopy_two}.

\begin{figure}[ht!]
\centering
\labellist
\small\hair 2pt
\pinlabel {R II} at 59 108
\pinlabel {R II} at 307 108
\endlabellist
\centering
\includegraphics[width=10cm, height=2.5cm]{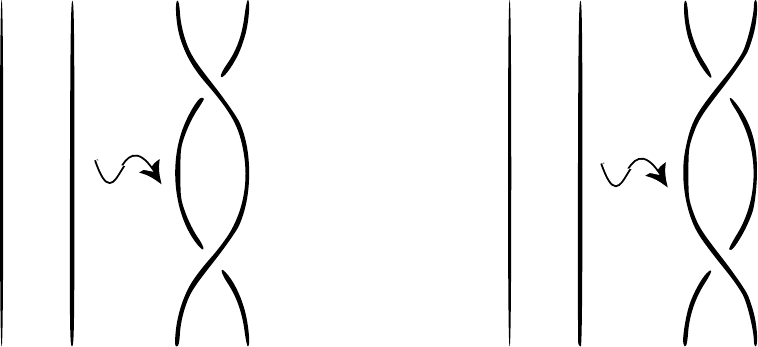}
\caption{Braid isotopy corresponding to Reidemeister II move}
\label{braidisotopy_one}
\end{figure}

\begin{figure}[ht!]
\centering
\labellist
\small\hair 2pt
\pinlabel {R III} at 125 110
\endlabellist
\centering
\includegraphics[width=6cm, height=2.5cm]{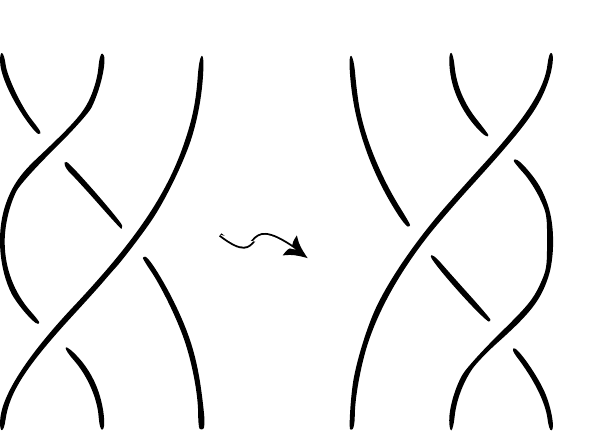}
\caption{Braid isotopy corresponding to Reidemeister III move}
\label{braidisotopy_two}
\end{figure}

\subsubsection{Stabilization and Debstabilization} \label{stab_destab}


\textbf{Stabilization}, shown in Figure~\ref{stab}, takes the plat defined by a $2n$-braid $\mbf{B}$ to the plat defined by the braid $\mbf{B} \sigma_{2n}$. Therefore, stabilization is equivalent to adding a generator and \textbf{destabilization} is the reverse process of that. Note that stabilisation increases the bridge index by one, while destabilisation decreases the bridge index by one.

\begin{figure}[ht!]
\centering
\labellist
\small\hair 2pt
\pinlabel {$\mbf{B}$} at 120 127
\pinlabel {$\mbf{B}$} at 428 127 
\pinlabel {$\sigma_{2n}$} at 529 26
\pinlabel {STABILIZATION} at 282 146
\pinlabel {DESTABILIZATION} at 282 100
\endlabellist
\centering
\includegraphics[width=18cm, height=7cm]{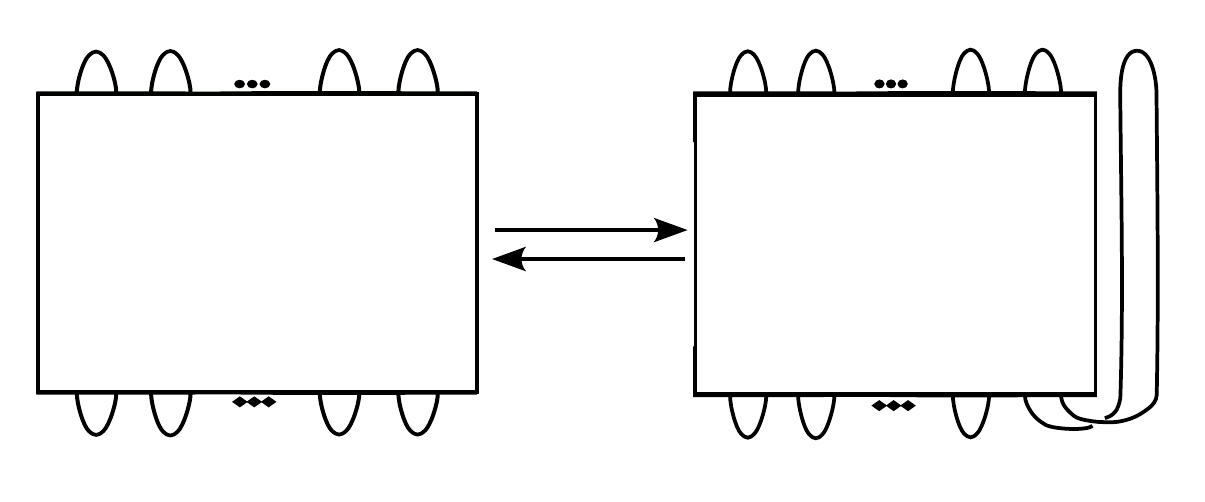}
\caption{Stabilizing and destabilizing a plat}
\label{stab}
\end{figure}
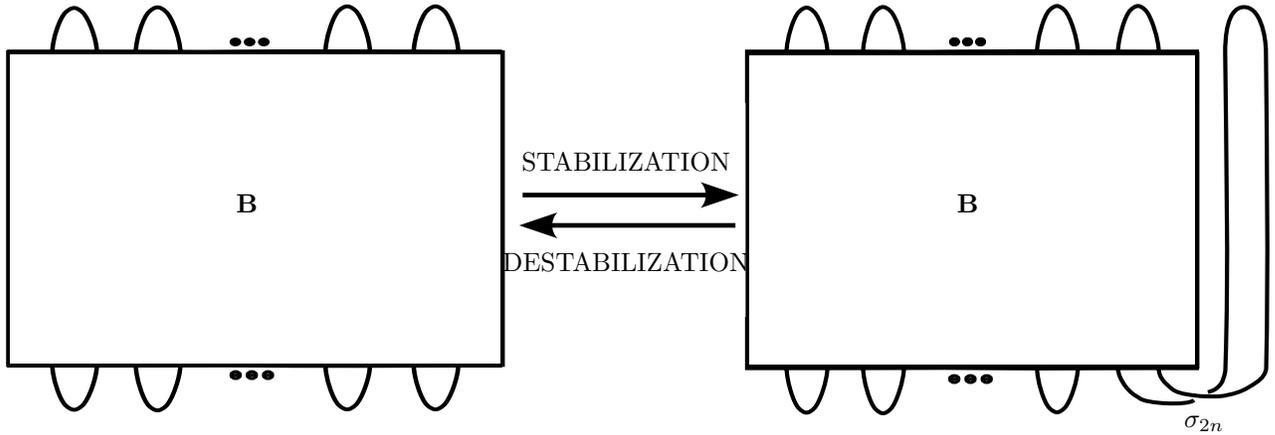


\subsubsection{Double coset moves and the Pocket move}

\label{dcms}

Double coset moves on plats permute the bridges in different ways, thus giving freedom of movement to the bridges. For the braid group $\mathcal{B}_{2n}$, a double coset move on the plat defined by $\sigma \in \mathcal{B}_{2n}$ is one which takes this plat to the plat defined by $\alpha \sigma \beta$  where $\alpha$ and $\beta$ are elements in $\mathcal{K}_{2n}$, the Hilden subgroup (\cite{hilden_sbgp}) of the braid group, which has the following generating set: 
$\{ \sigma_{1}, \sigma_{2} {\sigma_{1}}^{2} \sigma_{2}, \sigma_{2i} \sigma_{2i-1} \sigma_{2i+1} \sigma_{2i}, 1 \leq i \leq n-1 \}$. The relators for $\mathcal{K}_{2n}$ were computed by Stephen Tawn, in \cite{tawn2007presentation}. 

For $n=2$, the generators of the Hilden subgroup are: $\{ \sigma_{1}, \sigma_{2} {\sigma_{1}}^{2} \sigma_{2}, \sigma_{2} \sigma_{1} \sigma_{3} \sigma_{2} \}$. The plat moves on the bottom bridges coming from these generators are shown in Figure~\ref{hildengens}. 

Note that the double coset moves essentially try to capture all those moves, not covered by the scope of braid isotopies and stabilization/destabilization, which take plats to plats while preserving the link type. Figure~\ref{hildengens} is a good motivating example to observe this phenomenon. Notice that all the moves in Figure~\ref{hildengens}, change the braid word but not the knot type.

Birman, in \cite{birman_1976}, does not use the nomenclature, ``double coset moves". We are using it here, because, without this class of moves, given two plat presentations $K_1$ and $K_2$ of a knot, using \textbf{only} braid isotopies and stabilization/destabilization, the best we can do is get $K_1$ and $K_2$ to be in the same double coset of $\mathcal{B}_{2n}$ modulo $\mathcal{K}_{2n}$. 

\begin{figure}[ht!]
\centering
\labellist
\small\hair 2pt
\endlabellist
\centering
\includegraphics[width=12.5cm, height=3.5cm]{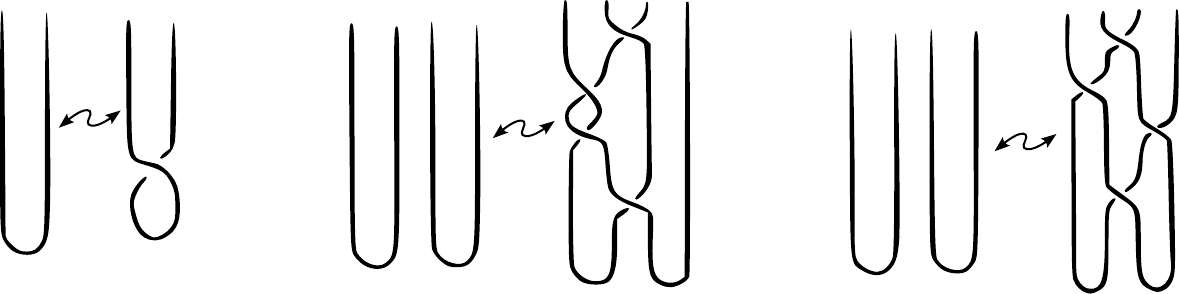}
\caption{Generators of $\mathcal{K}_4$}
\label{hildengens}
\end{figure}





\begin{figure}[ht!]
\centering
\labellist
\small\hair 2pt
\pinlabel $A$ at 92 177 
\pinlabel $A$ at 153 182
\pinlabel $A$ at 248 181
\pinlabel $A$ at 323 182
\pinlabel $B$ at 92 7
\pinlabel $B$ at 155 11
\pinlabel $B$ at 248 11
\pinlabel $B$ at 325 11
\endlabellist
\centering
\includegraphics[width=10cm, height=4cm]{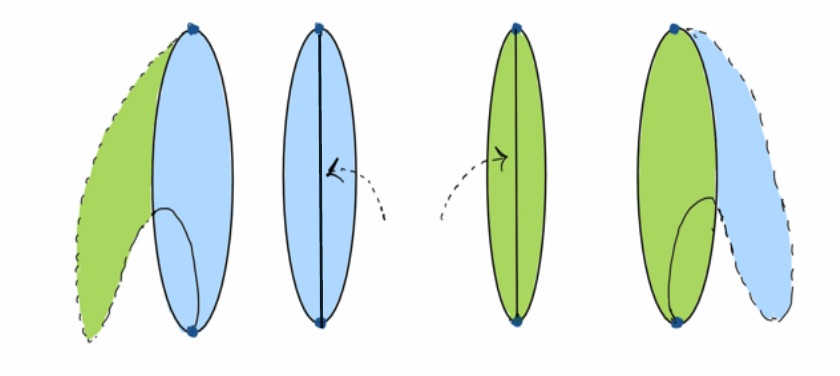}
\caption{The pocket move on the index one unknot}
\label{pocket}
\end{figure}

\begin{figure}[ht!]
\centering
\labellist
\small\hair 2pt
\pinlabel {A} at 224 402
\pinlabel {B} at 224 118
\pinlabel {A} at 585 402
\pinlabel {B} at 585 118
\endlabellist
\centering
\includegraphics[width=15cm, height=8.3cm]{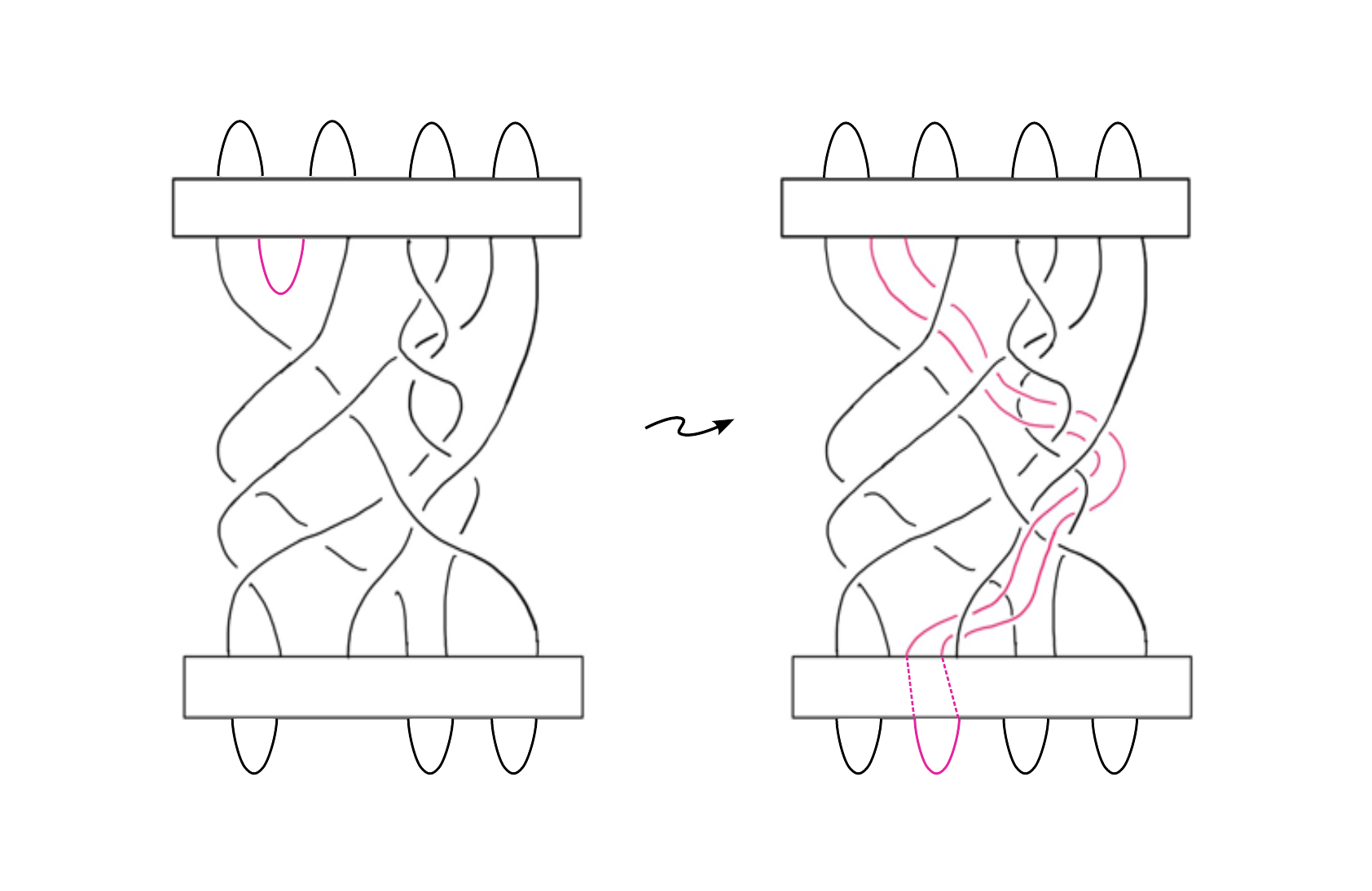}
\caption{pocket move on a higher index plat}
\label{pocket_general}
\end{figure}

\begin{figure}[ht!]
\centering
\labellist
\small\hair 2pt
\pinlabel {A} at 153 361
\pinlabel {B} at 143 85
\pinlabel {A} at 444 361
\pinlabel {B} at 458 85
\pinlabel {C} at 118 313
\pinlabel {C} at 423 39
\endlabellist
\centering
\includegraphics[width=15cm, height=8.5cm]{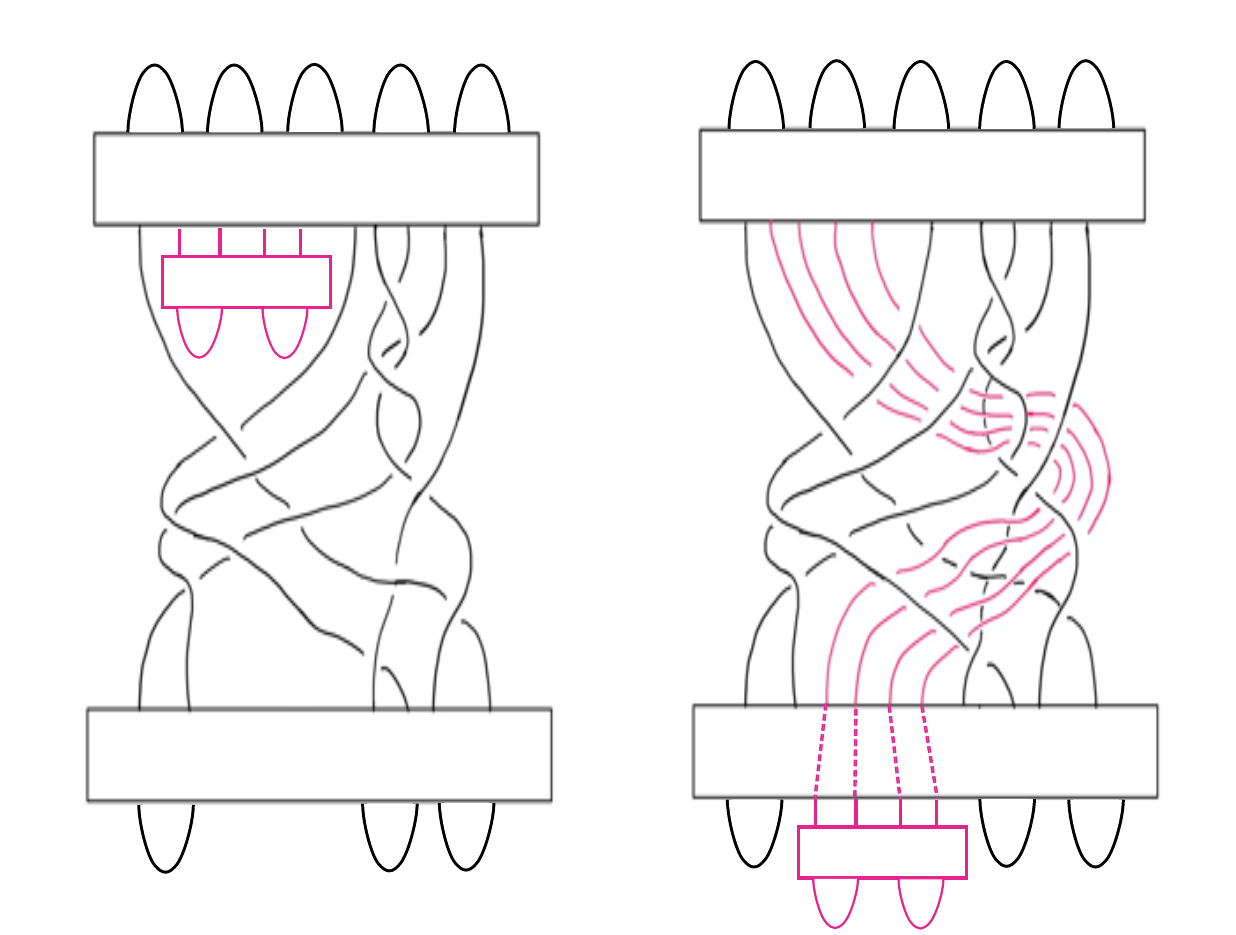}
\caption{Generalised pocket move on a higher index plat}
\label{block_pocket}
\end{figure}


Figure \ref{pocket} shows part of the motivation behind the pocket move. The original idea was to capture moves which are a confluence of the following two phenomenon: \\
(i) plat moves which create saddles on the spanning disc of the unknot \\
(ii) plat moves that decode what the double cosets moves achieve geometrically


Figures \ref{pocket_general} and \ref{block_pocket} depict a generic pocket move. Consider a plat presentation where a top or a bottom bridge which is truncated at a level strictly between the top and the bottom levels of the plat (as shown in the Figures), say $t = t_0$, where $0 < t_0 < 1$. Now, define a smooth path $\gamma$, $\gamma: [0,1] \lra \mbb{R}^3$, such that the plat height function is monotonic on $\gamma$ and we have that $\gamma(0) = t_0$ and $\gamma(1) = 1$. Now, we take the local max or min corresponding to the bridge that we picked and drag it along the path $\gamma$. This results in the bridge weaving around the rest of the plat, as shown in Figure \ref{pocket_general}. Further, we can also twist the bridge as it follows the path $\gamma$, and we can do this process for any top or bottom bridge that can be truncated prematurely, with each bridge following a potentially different path.

\begin{lemma}
   
Pocket moves can be realised as a combination of braid isotopies and a sequence of double coset moves.

\label{pmisdcm}
\end{lemma}

\begin{proof}

Without loss of generality, consider a pocket move where the first bottom bridge from the left follows a path $\alpha$ and to cause braiding with the rest of the plat. The logic is similar for any other top or bottom bridge. If there are multiple bridges involved in the pocket move, an induction argument will suffice. \\
Then, we consider the position of this bridge before and after doing the pocket move, as the example in Figure \ref{pmdcm} illustrates. There we are doing the pocket move using the first bridge from the left, considered at the top. The position of the bridge after the pocket move is shown in pink, and before doing the pocket move is in black. Notice that the pink arc and the black arc co-bound a disc. We consider the level curves of this disc. For this disc, at each level other than the top and bottom level, we will see two pink dots and two black dots, and two arcs joining these four dots. Notice that there will be a singularity at some level where these two arcs intersect. Before this level, we have two arcs, each joining a pink dot to a black dot, after the level of singularity, we have two arcs, one joining the two black dots to each other and another one joining the pink dots. We push this singularity down to a level $t_{\alpha}$ such that all the crossings in the plat are at a level strictly above $t_{\alpha}$. Below level $t_{\alpha}$, call the arc joining the two black dots, the ``black arc''. This process of pushing down the singularity causes the black arc to get coiled up, in order to avoid the spanning disc of the unknot at any level. Uncoiling this black arc below level $t_{\alpha}$ can be achieved through doing double coset moves, and this process will isotope the black bottom bridge to its intended configuration of the pink bridge. This can be observed in Figure \ref{pmdcm}, where we illustrate how the pocket move done anywhere can be moved to a level below any other crossings in the plat, using braid isotopies, as shown in picture IV. Once we reach the configuration in picture IV, it is clear how we can go from the black bridge to the pink one using just the double coset moves. \\


\begin{figure}[ht!]
\centering
\labellist
\pinlabel $I$ at 103 293 
\pinlabel $II$ at 285 293
\pinlabel $III$ at 90 10
\pinlabel $IV$ at 302 10
\endlabellist
\centering
\includegraphics[width=12cm, height=11cm]{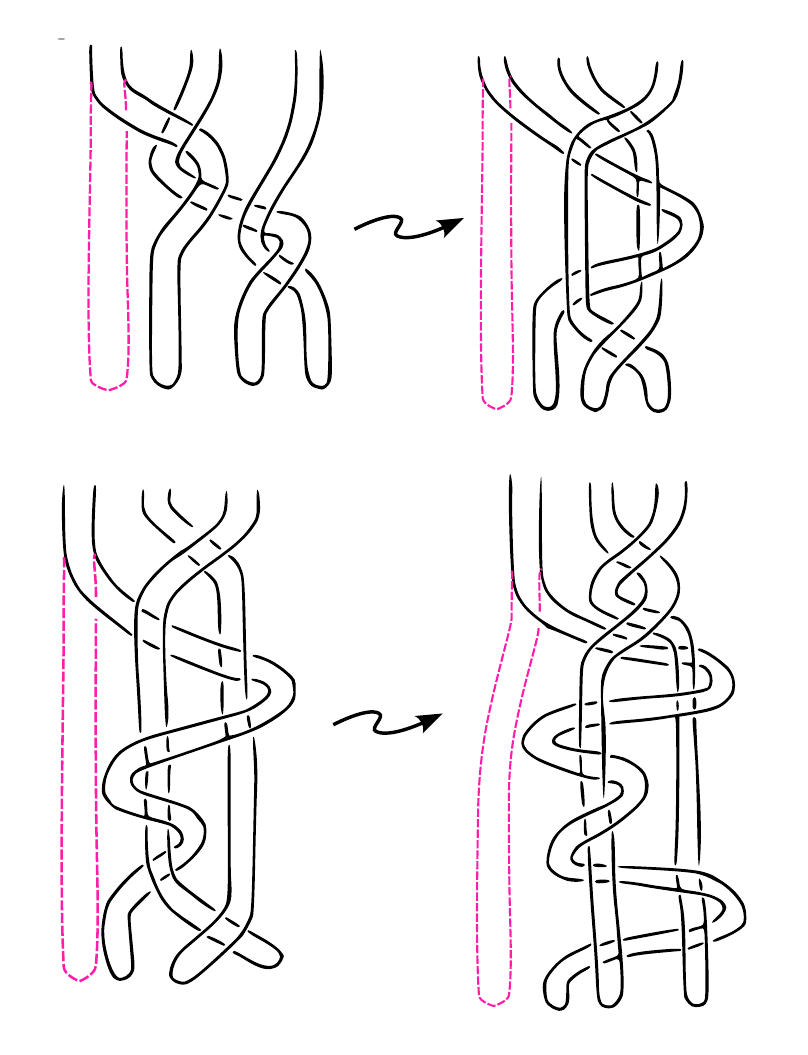}
\caption{Undoing a pocket move using double coset moves}
\label{pmdcm}
\end{figure}
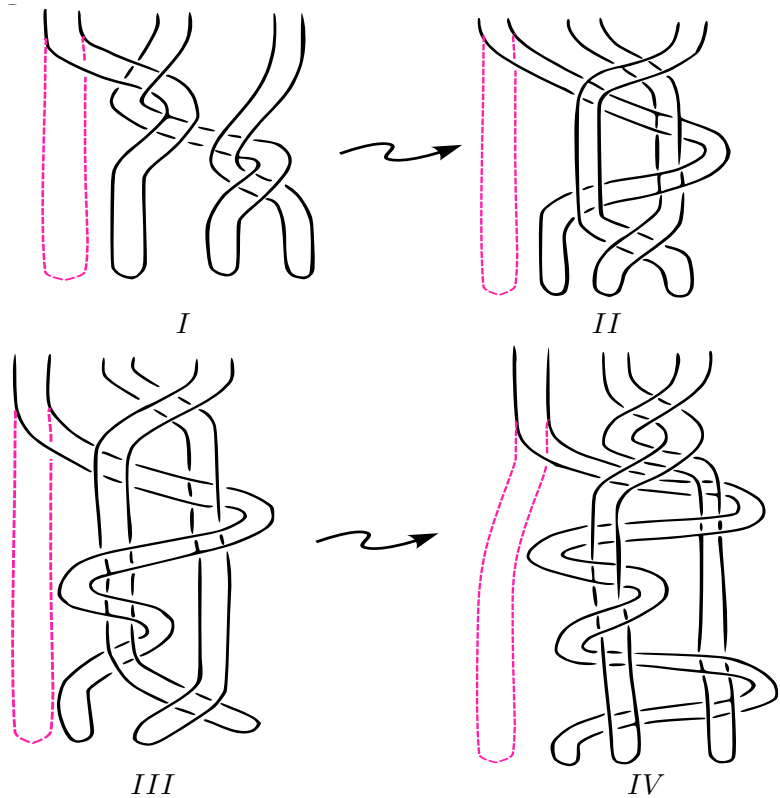

\end{proof}

\subsubsection{The flip move}

\label{flipmove}

Figure ~\ref{flip} shows the \textbf{flip move on the index one unknot diagram}, which is the following isotopy (or its inverse): take point B ($(0,-2,0)$) and move it along the following path: 
$$ x = 0, y = \cos(\theta) -1, z = \sin(\theta)$$ where $ \pi \leq \theta \leq 3\pi$. 


This path traces out the circle of radius one centered at $(0,-1,0)$, in the $y-z$ plane, moving clockwise if viewed from the region $x \geq 0$. The effect of this isotopy on the disc is shown in Figure~\ref{flip}.

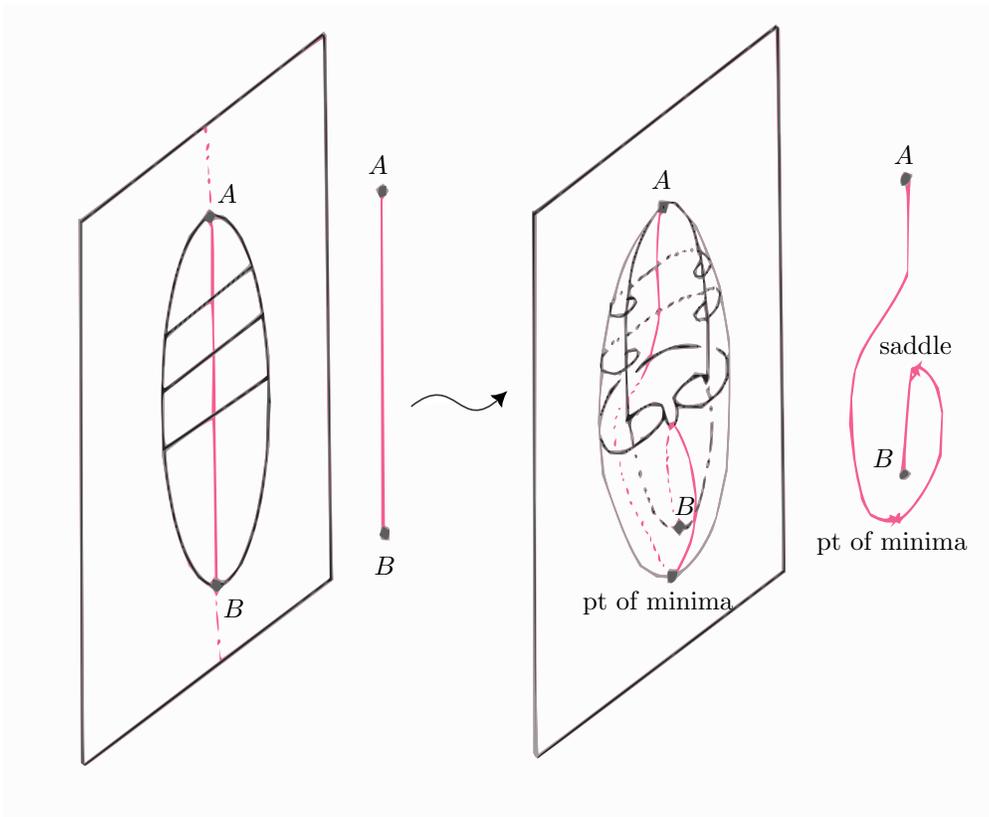
\begin{figure}[ht!]
\centering
\labellist
\small\hair 2pt
\pinlabel $A$ at 132 368 
\pinlabel $A$ at 221 385
\pinlabel $A$ at 388 376
\pinlabel $A$ at 531 391
\pinlabel $B$ at 136 124
\pinlabel $B$ at 225 149
\pinlabel $B$ at 402 184
\pinlabel $B$ at 519 212
\pinlabel {saddle} at 538 279
\pinlabel {pt of minima} at 386 127
\pinlabel {pt of minima} at 524 162
\endlabellist
\centering
\includegraphics[scale=0.65]{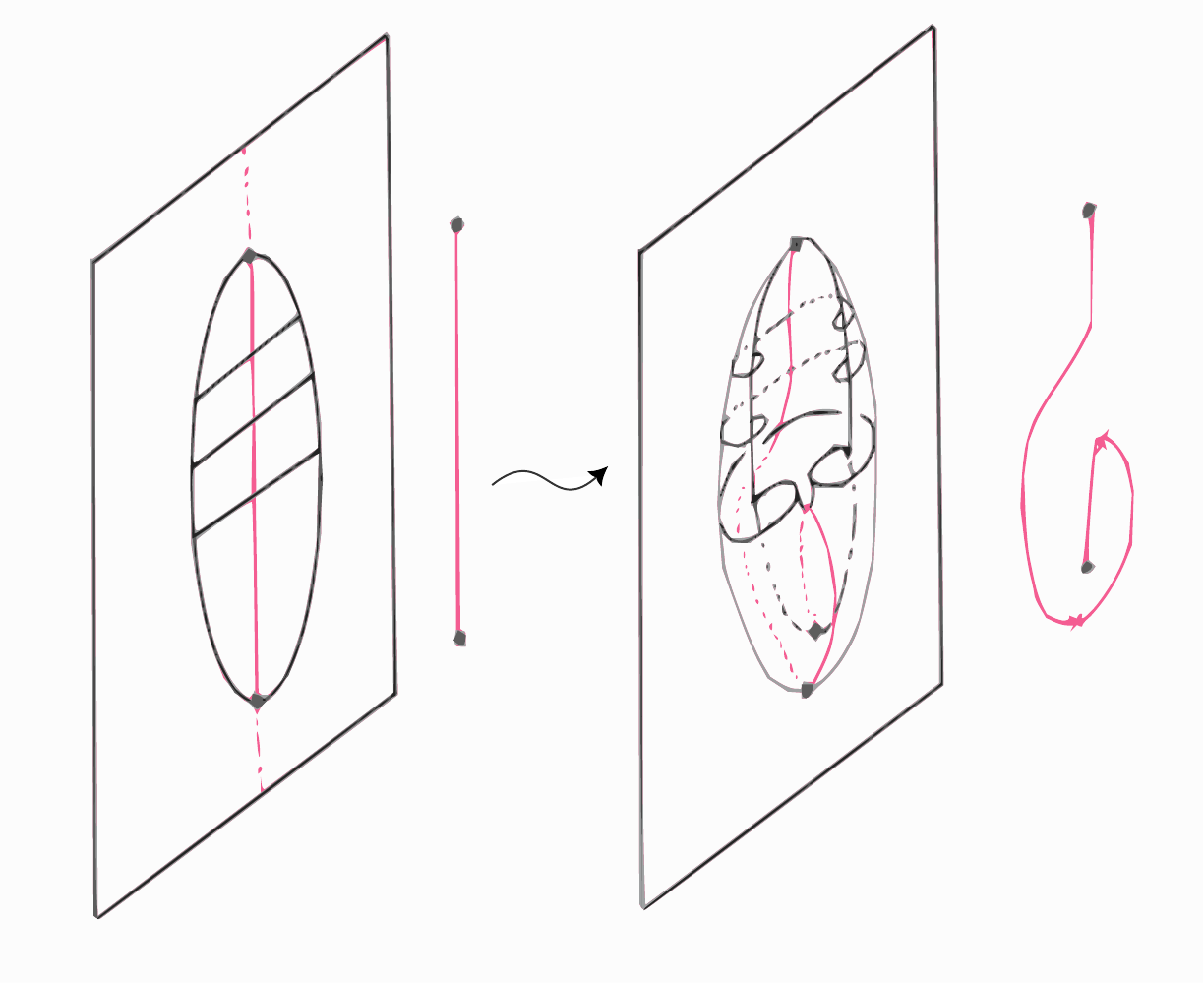}
\caption{The flip move on the 1-bridge 0-crossing diagram}
\label{flip}
\end{figure}

\begin{figure}[ht!]
\centering
\labellist
\small\hair 2pt
\pinlabel $A$ at 154 311 
\pinlabel $A$ at 472 311
\pinlabel $B$ at 157 102
\pinlabel $B$ at 474 102
\endlabellist
\centering
\includegraphics[width=12cm, height=6.5cm]{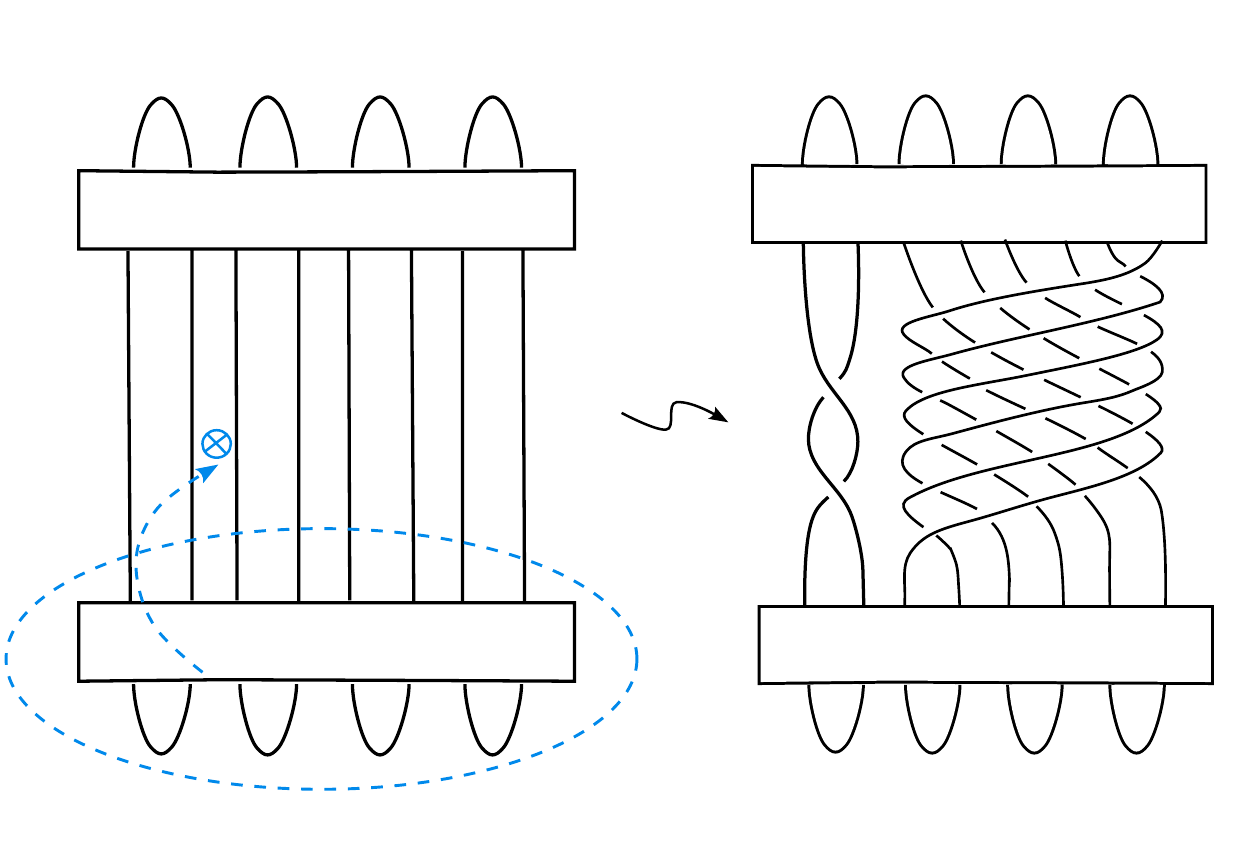}
\caption{The flip move on an 8 plat}
\label{flip_8}
\end{figure}

On a generic $2n$-plat with braid word $\mathbf{AB}$, here's how the flip move changes the braid word: \\
Case (i): Flipping $\mathbf{B}$ into the plane of the paper between the $k$-th and $(k+1)$-st strands: 
$$ \mathbf{AB} \lra \mathbf{A} {(\sigma_1 \sigma_2...\sigma_{k-1})}^{k} {((\sigma_{2n-1})^{-1}...(\sigma_{k+2})^{-1}(\sigma_{k+1})^{-1})}^{2n-k} \mathbf{B} $$ 
Case (ii): Flipping $\mathbf{B}$ out of the plane of the paper between the $k$-th and $(k+1)$-st strands: 
$$ \mathbf{AB} \lra \mathbf{A} {((\sigma_{k-1})^{-1}(\sigma_{k-2})^{-1}...(\sigma_{1})^{-1})}^{k} {(\sigma_{k+1} \sigma_{k+2}... \sigma_{2n-1})}^{2n-k} \mathbf{B} $$ 
Case (iii): Flipping $\mathbf{B}$ into the plane of the paper between the first two strands: 
$$ \mathbf{AB} \lra \mathbf{A} {((\sigma_{2n-1})^{-1} (\sigma_{2n-2})^{-1}... (\sigma_{3})^{-1} (\sigma_{2})^{-1})}^{2n-1} \mbf{B} $$ 
Case (iv): Flipping $\mbf{B}$ out of the plane of the paper between the first two strands: 
 $$ \mbf{AB} \lra \mbf{A} {(\sigma_2 \sigma_3... \sigma_{2n-2} \sigma_{2n-1})}^{2n-1} \mbf{B} $$
Case (v): Flipping $\mbf{B}$ into the plane of the paper between the last two strands: $$ \mbf{AB} \lra \mbf{A} {(\sigma_1 \sigma_2 ... \sigma_{2n-3} \sigma_{2n-2})}^{2n-1} \mbf{B} $$
Case (vi): Flipping $\mbf{B}$ out of the plane of the paper between the last two strands: 
$$ \mbf{AB} \lra \mbf{A} {((\sigma_{2n-2})^{-1} (\sigma_{2n-3})^{-1} ... (\sigma_{2})^{-1} (\sigma_{1})^{-1})}^{2n-1} \mbf{B} $$ 

Further, we define the \textbf{micro flip move}, wherein only $k$ (where $k$ is even) out of the $n$ strands are flipped. A pictorial representation is given in Figure~\ref{microflip}.

\begin{figure}[ht!]
\centering
\labellist
\small\hair 2pt
\pinlabel $A$ at 134 283
\pinlabel $A$ at 475 282
\pinlabel $B$ at 82 85
\pinlabel $B$ at 429 84
\endlabellist
\centering
\includegraphics[width=13.5cm, height=6cm]{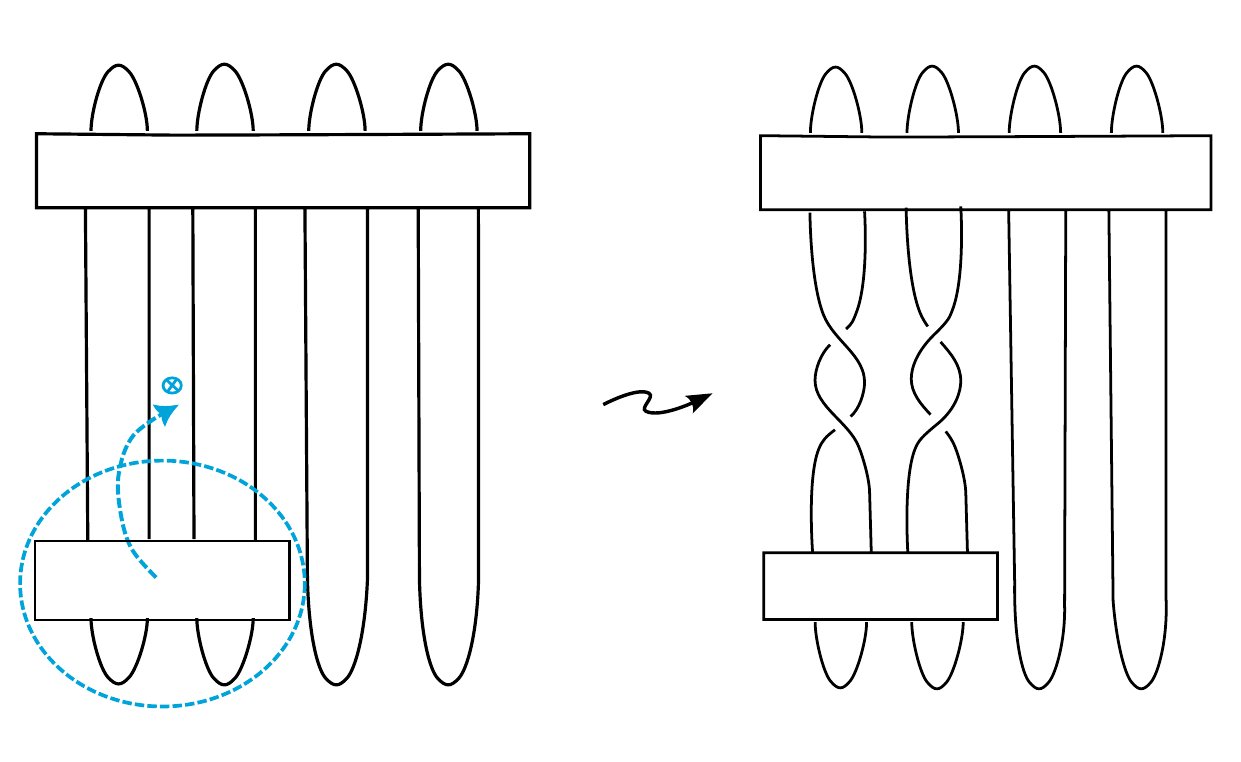}
\caption{An example of a Microflip move}
\label{microflip}
\end{figure}
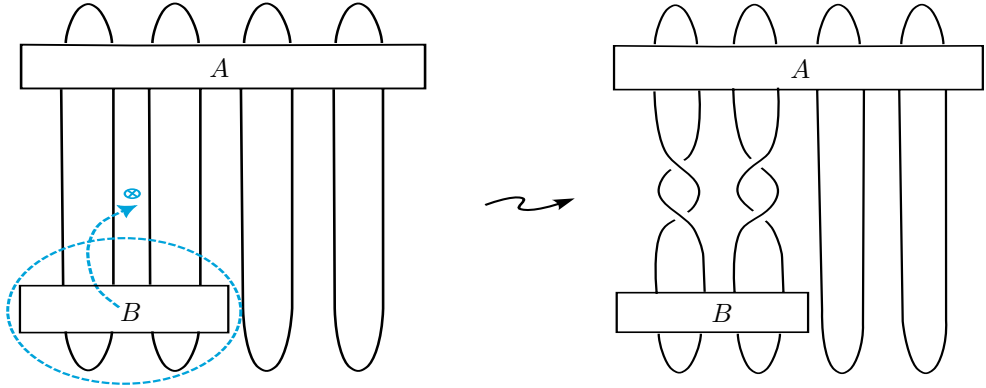

Corresponding to the aforementioned cases, we also have six more cases where we do the flip with the top bridges. The change in braid word can be computed as shown above.
Figure~\ref{stab_seq} shows the corresponding sequence of moves using only braid isotopies, double coset moves and stabilization/destabilization, to go from the plat on the right to the plat on the left in Figure~\ref{flip_8}. Repeating the same sequence of moves on the first strand (from the left) in the last plat in Figure~\ref{stab_seq}, we get the left plat in the first row in Figure~\ref{garside_slide}. This plat is now defined by the braid $A \delta B$, where $\delta$ is the full twist in the braid group $\mathcal{B}_{2n}$, which generates the center of the braid group. Therefore, via braid isotopies, we can slide the braid word $B$ across the full twist to get the second plat, defined by $AB\delta$, in Figure~\ref{garside_slide}. Now, we can see that this plat can be resolved via double coset moves to give us the plat defined by $AB$, that we wanted to get to. 

The flip move allows us to turn different strands of the braid both ways simultaneously, which we can not achieve with just the double coset moves.


\begin{figure}[ht!]
\centering
\labellist
\small\hair 2pt
\pinlabel $A$ at 120 513
\pinlabel $A$ at 361 515
\pinlabel $A$ at 603 515
\pinlabel $A$ at 106 233
\pinlabel $A$ at 357 233
\pinlabel $A$ at 624 233
\pinlabel $B$ at 120 358
\pinlabel $B$ at 361 356
\pinlabel $B$ at 603 357
\pinlabel $B$ at 106 74
\pinlabel $B$ at 357 75
\pinlabel $B$ at 624 73
\endlabellist
\centering
\includegraphics[width=18cm, height=13cm]{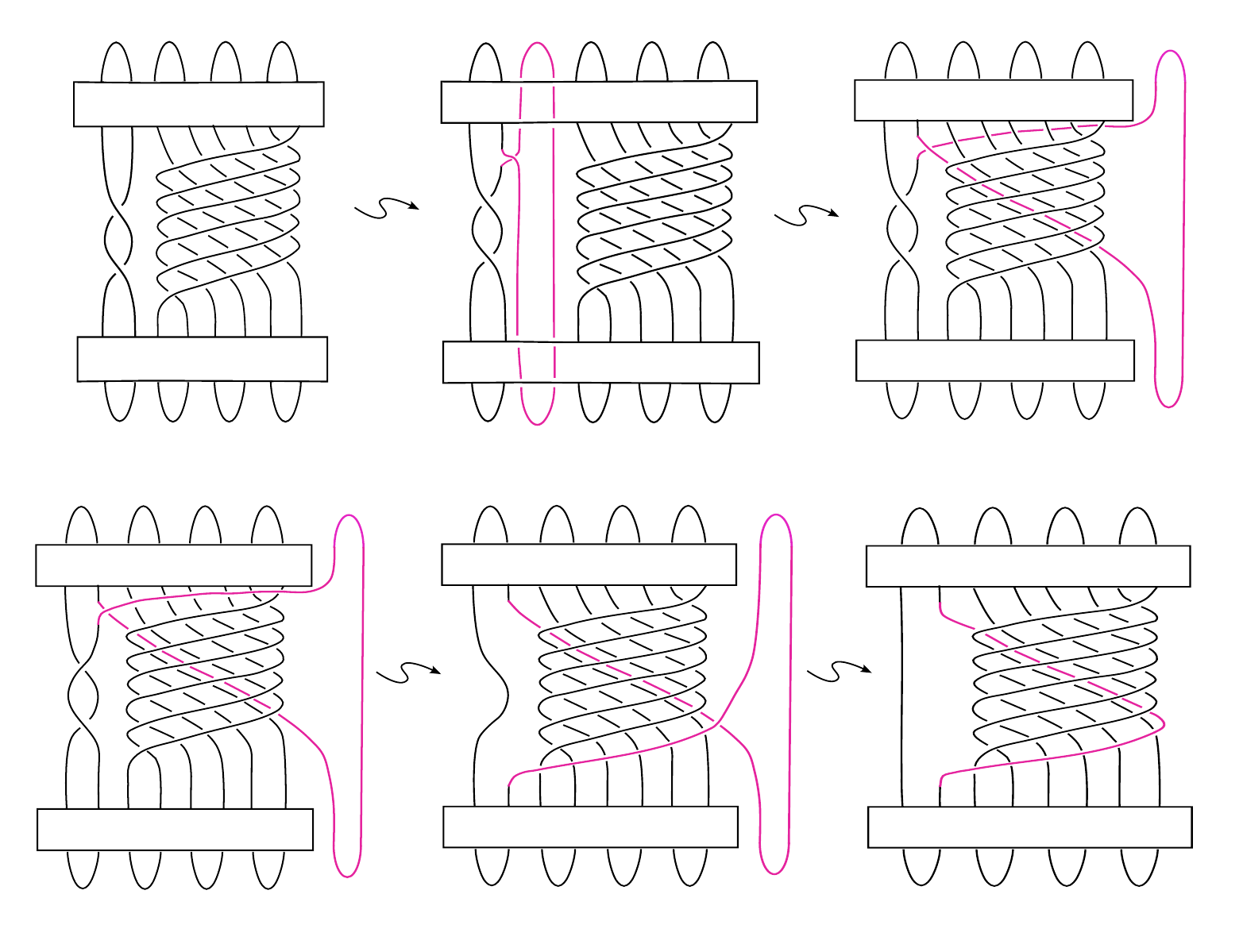}
\caption{Stabilization sequence for the flip move}
\label{stab_seq}
\end{figure}
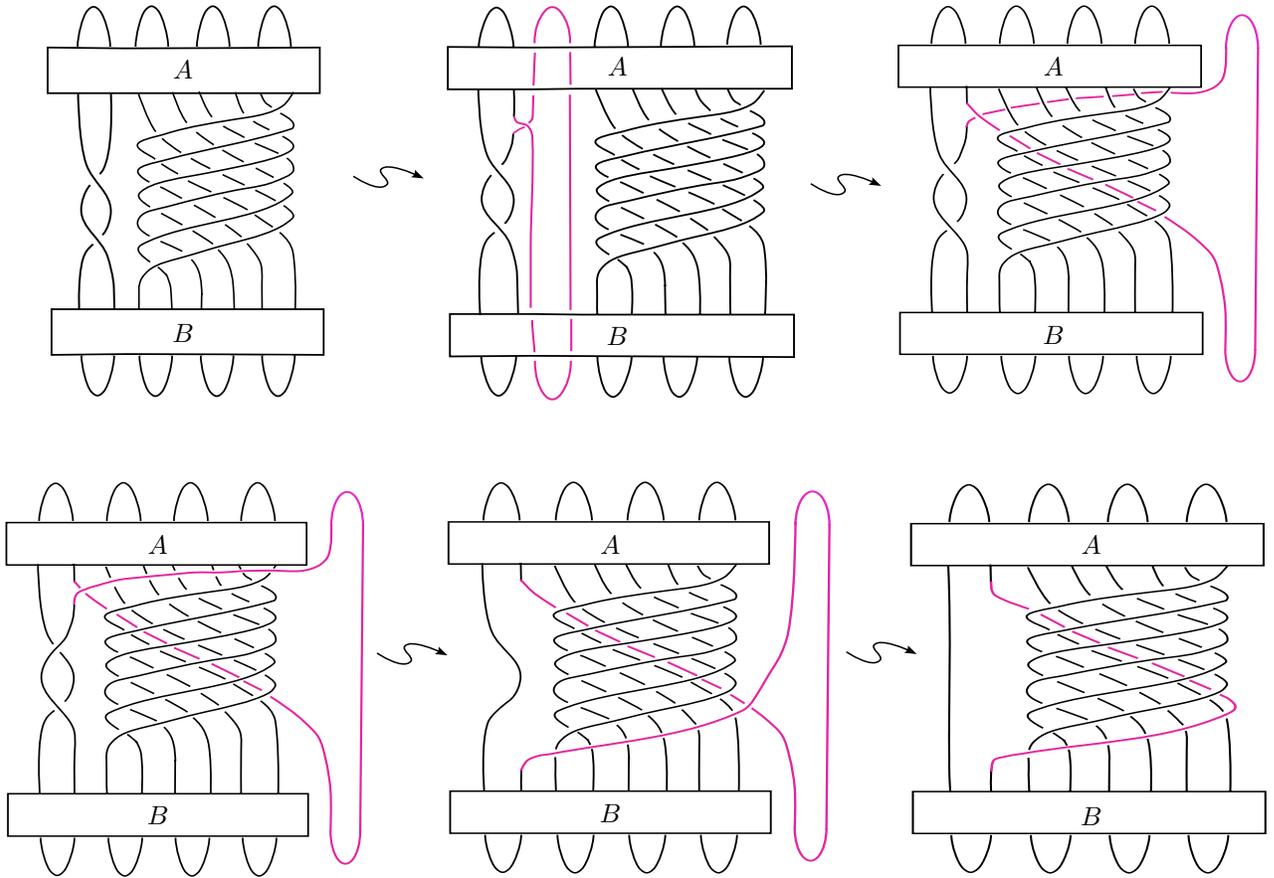

\begin{figure}[ht!]
\centering
\labellist
\small\hair 2pt
\pinlabel $A$ at 139 683
\pinlabel $A$ at 296 292
\pinlabel $B$ at 140 457
\pinlabel $B$ at 293 95
\pinlabel $AB$ at 436 683
\endlabellist
\centering
\includegraphics[width=14.5cm, height=14.5cm]{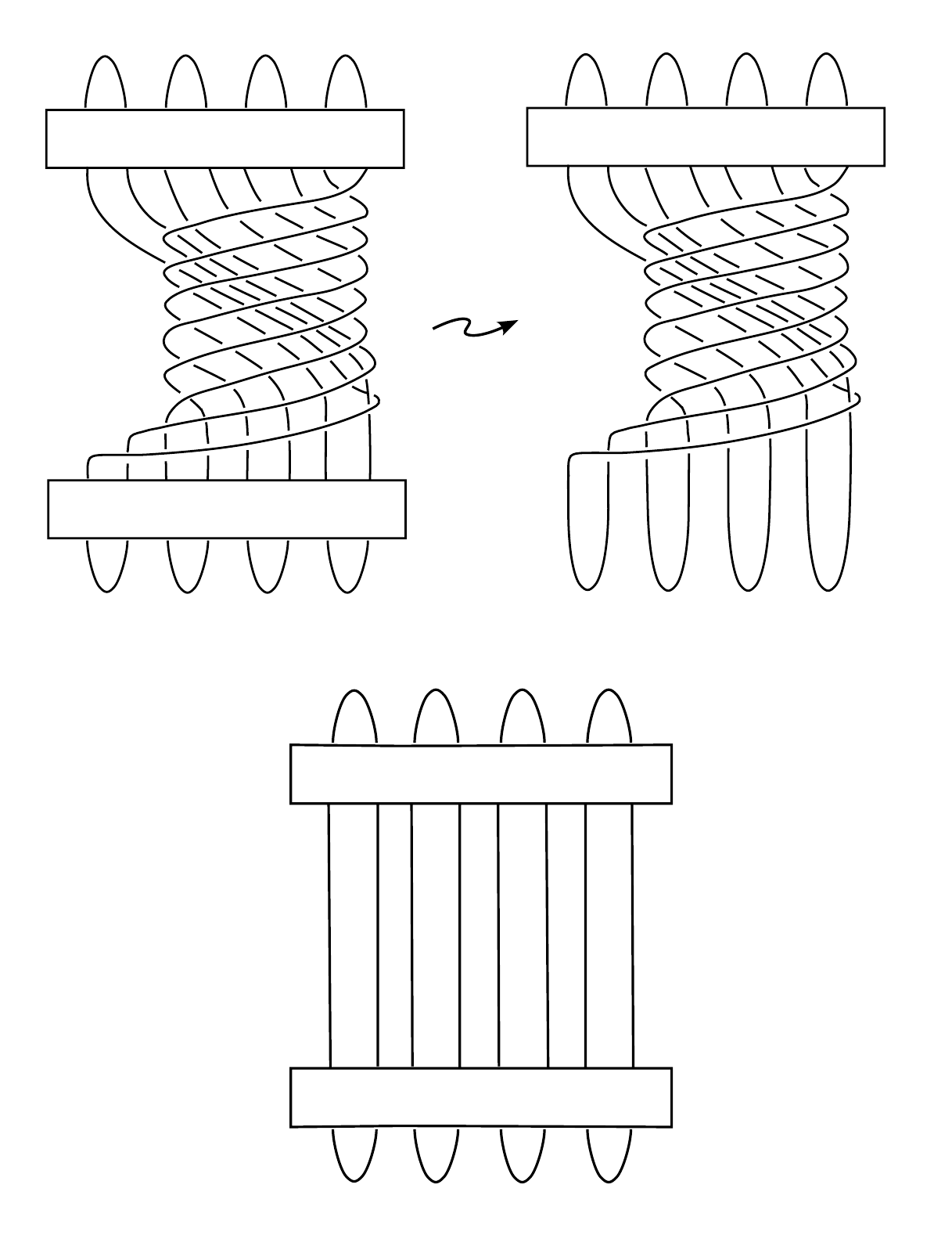}
\caption{Using double coset moves to undo the `turning' of a braid}
\label{garside_slide}
\end{figure}

\subsection{Disc Foliation and Tiling} \label{tilingandfoliation}


As we remarked in \S \ref{level curves}, since the level curves $C_t^i$ are cross sections of $D_i$, taking a union of $C_t^i$'s over all $t$'s gives us the surface $D_i$. Therefore, ${D_i} = \bigsqcup_{t \in [-0.25,1.25]} C_t^i$. This gives a singular foliation of the surface $D_i$, where the leaves of the foliation are arcs $a_{i,t}^j$'s and curves $s_{i,t}^j$'s. The singular leaves are exactly the  intersecting (which includes self intersections) arcs and curves, of which there are only finitely many. The singular leaves correspond to saddles on the surface. Further, self intersecting arcs or closed curves give rise to curves which either self intersect again or  degenerate to a single point in the foliation, which corresponds to a local maxima or local minima on $D_i$. If $\mbf{K_i}$ is a plat of index $n$, then there are $n$ points of maxima and $n$ points of minima on the boundary of $D_i$, which is $\mbf{K_i}$. As mentioned earlier, these points of maxima are at height $t =1$ and the minima are at height $t=0$. 

Using this singular foliation, we will now describe a way to tile $D_i$. We thicken up the singular leaves such that the union of the `neighborhoods' of the singular leaves then determines a tiling of $D_i$. So, a tile $T$ on the disc is a thickened up neighborhood of a singular leaf $s$, i.e. a tile $T$ is the surface $s \times I$, where $I$ is the unit interval. Thus, each tile on the disc contains exactly one singularity (saddle or local min/max). Any tile $T$ has a boundary comprising of arcs and circles, where the arcs might lie on the boundary of the surface (i.e. the unknot) or in the interior of $D_i$. A tile which has $a$ edges lying on the knot, $b$ lying in the interior of the surface, and $c$ boundary circles will be denoted as: \bm{$T_{(a,b,c)}$}. 

Below we give a detailed description of all the possible singularities (and hence, tiles) which show up in the foliation (and tiling) of $\mbf{D}$. In the figures referenced below, we follow the convention that for the tiles, black line segments on the boundary of the tile are the ones lying on the unknot and the blue line segments or simple closed curves on the boundary of the tile are the ones lying in in the interior of the spanning disc. Keeping in line with this convention, for the local 3-D renderings of the tiles also, black line segments represent the unknot. Also, note that each tile is foliated by the level curves $C_t^i$, in a manner so that each tile contains exactly one self intersecting singular leaf or exactly one pair of intersecting singular leaves. With that in mind, we have the following tiles: 

(i) \textbf{Intersection of two arcs}: Figure~\ref{alpha} shows this type of singularity, with the first row describing the level curves associated with it and the second row shows the tile and the local 3-D embedding of the spanning disc associated with it. We call this singularity, and abusing notation, also the tile generated by it, \bm{$T_{(4,4,0)}$}. The tile \bm{$T_{(4,4,0)}$} is a disc with 8 edges, 4 of which lie on the knot (i.e. the boundary of $\mbf{D}$), denoted by black edges in the tile, and the other 4 edges lie in the interior of the disc, glued to the edges of tiles coming from the other singularities. The black straight lines in the local 3-D embedding depict the unknot.

Note that, for a generic $n$-bridge diagram, we could conceivably have more than two arcs intersecting at the same level but we can get rid of such a saddle via an isotopy so that at any level we have no more than one singularity. We can do this because the height function on $\mbf{D}$ is a Morse function.





\begin{figure}[h!]
\labellist
\small\hair 2pt
\pinlabel {before saddle}  at 189 393 
\pinlabel {at level of saddle} at 467 393
\pinlabel {after saddle} at 732 393
\pinlabel {tile $\bm{T_{(4,4,0)}}$} at 190 63
\pinlabel {3-D embedding of $\bm{T_{(4,4,0)}}$} at 638 33
\endlabellist
\centering
\includegraphics[width=10cm, height=7cm]{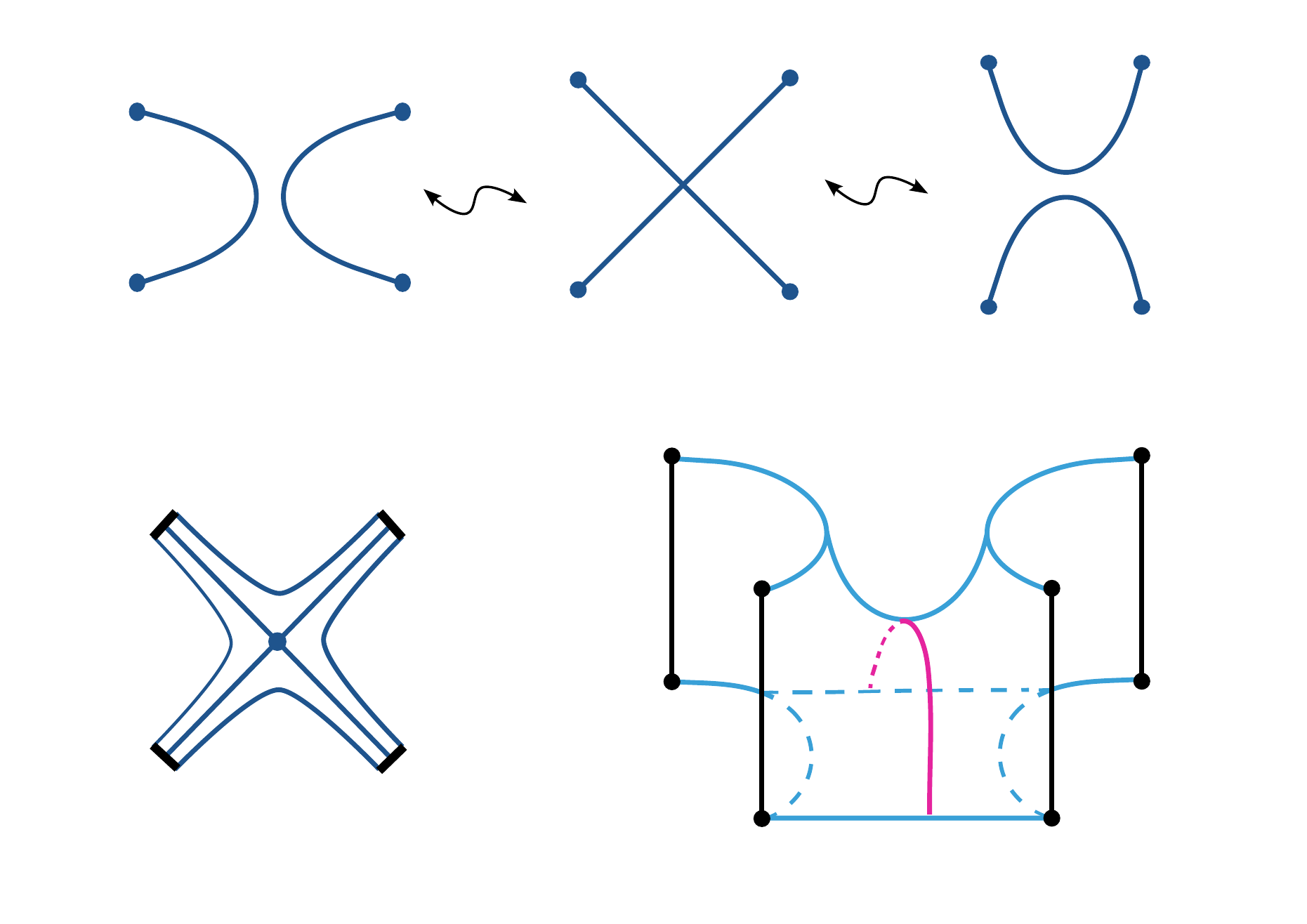}
\caption{Intersection of two arcs - $\bm{T_{(4,4,0)}}$}
\label{alpha}
\end{figure}
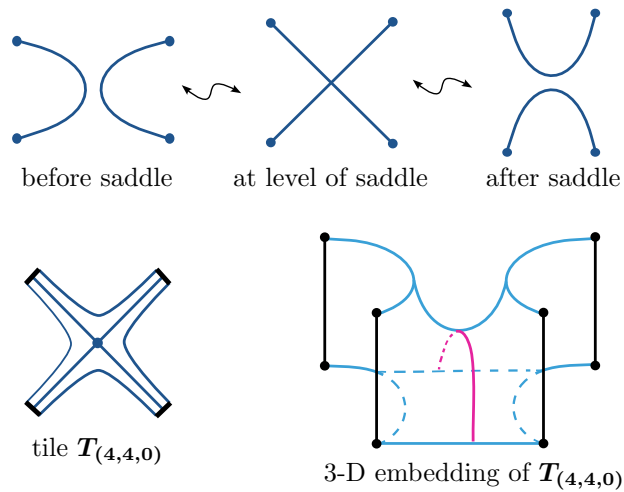

(ii) \textbf{Arc intersecting itself (or arc intersecting a circle)}: This is described in Figure~\ref{beta} and we call this singularity (and the respective tile it generates, as before) \bm{$T_{(2,2,1)}$}. The first row of level curves is corresponding to the pocket move and the second row corresponds to the flip move. This is a key difference between the two moves, other differences, although they do show up in the foliation, are very subtle and thus harder to recognise. The last row shows the two possible local 3-D embeddings of the spanning disc associated with \bm{$T_{(2,2,1)}$}, the one on the left corresponds to the pocket move and the first row of level curves, and the second one corresponds to the flip move and the second row of level curves. 

The tile \bm{$T_{(2,2,1)}$} is an annulus with 4 edges and one boundary circle, such that 2 edges (colored black in the tile in Figure~\ref{beta}) lie on the knot, and the other 2 lie in the interior of $\mbf{D}$, glued to the edges of some other tiles. As before, the straight black lines in the local 3-D embedding denote the unknot and the green region in the tile is {\em not} part of the tile.





\begin{figure}[h]
\labellist
\small\hair 2pt
\pinlabel {before saddle}  at 98 165 
\pinlabel {at level of saddle} at 289 165
\pinlabel {after saddle} at 446 165
\pinlabel {tile \bm{$T_{(2,2,1)}$}} at 100 17
\pinlabel {3-D embeddings of \bm{$T_{(2,2,1)}$}} at 371 17

\endlabellist
\centering
\includegraphics[width=12cm, height=9cm]{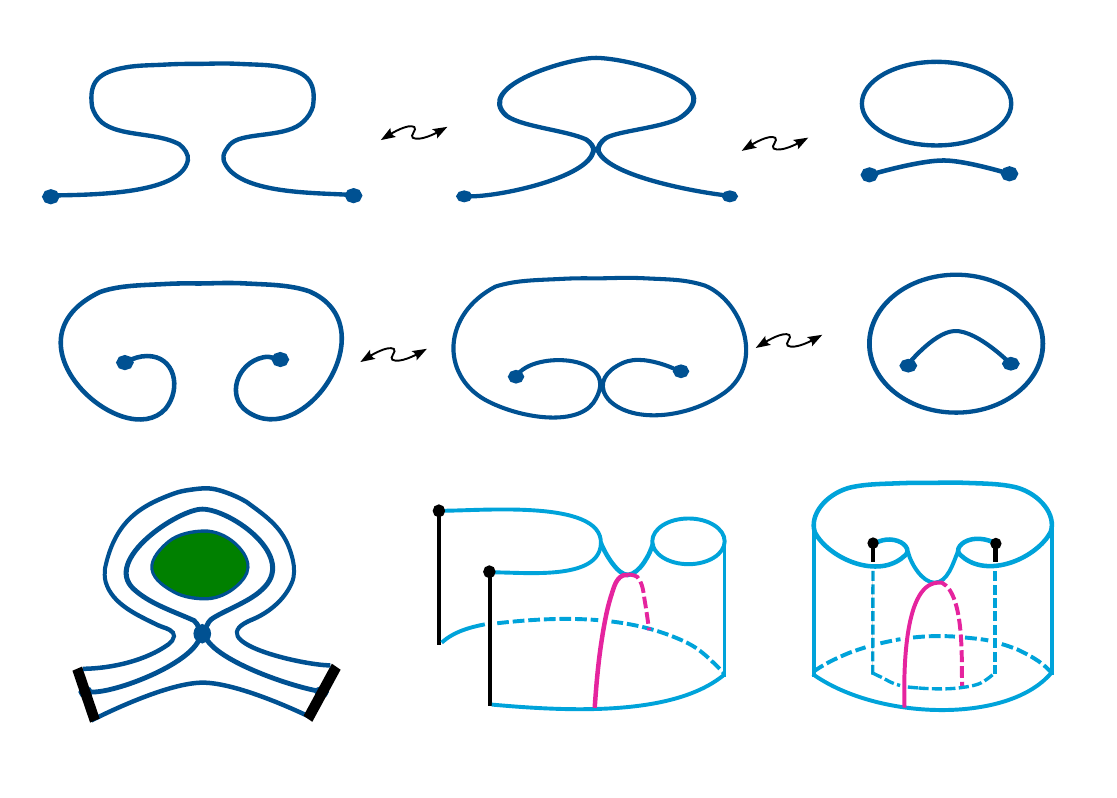}
\caption{Arc intersecting itself - \bm{$T_{(2,2,1)}$}}
\label{beta}
\end{figure}
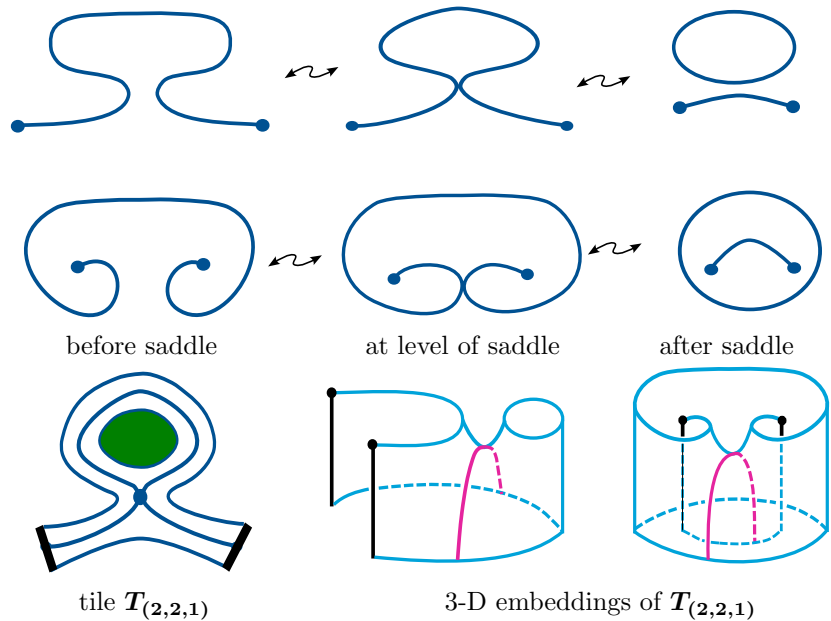

(iii) \textbf{Circle intersects circle}: Figure~\ref{gamma} depicts the singularity (and the tile generated by it) \bm{$T_{(0,0,3)}$}. The first row of level curves comes from a nested sequence of two pocket moves and the second one comes from a { \em pocket move} followed by a flip move. As before, the top row corresponds to the 3-D embedding on the left and the bottom row corresponds to the 3-D embedding on the right. The tile \bm{$T_{(0,0,3)}$} is a pair of pants with three boundary circles, all in the interior of the disc. The green shaded portion shown in Figure~\ref{gamma} is NOT a part of the tile. 





\begin{figure}[h!]
\labellist
\small\hair 2pt
\pinlabel {before saddle}  at 134 208
\pinlabel {at level of saddle} at 368 208
\pinlabel {after saddle} at 607 210
\pinlabel {tile \bm{$T_{(0,0,3)}$}} at 136 32
\pinlabel {3-D embeddings of \bm{$T_{(0,0,3)}$}} at 492 32
\endlabellist
\centering
\includegraphics[width=14cm, height=9cm]{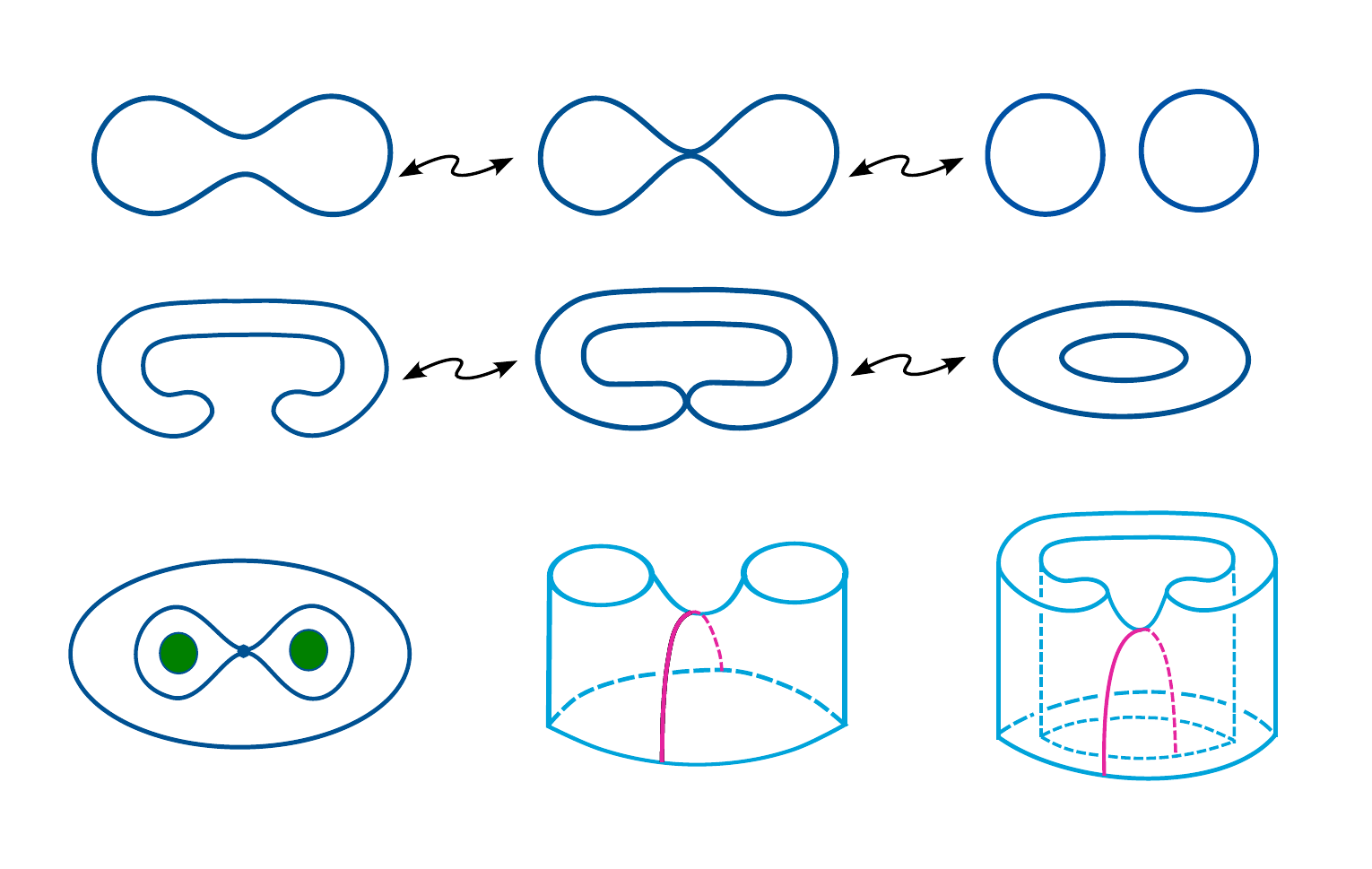}
\caption{Circle intersecting itself - \bm{$T_{(0,0,3)}$}}
\label{gamma}
\end{figure}
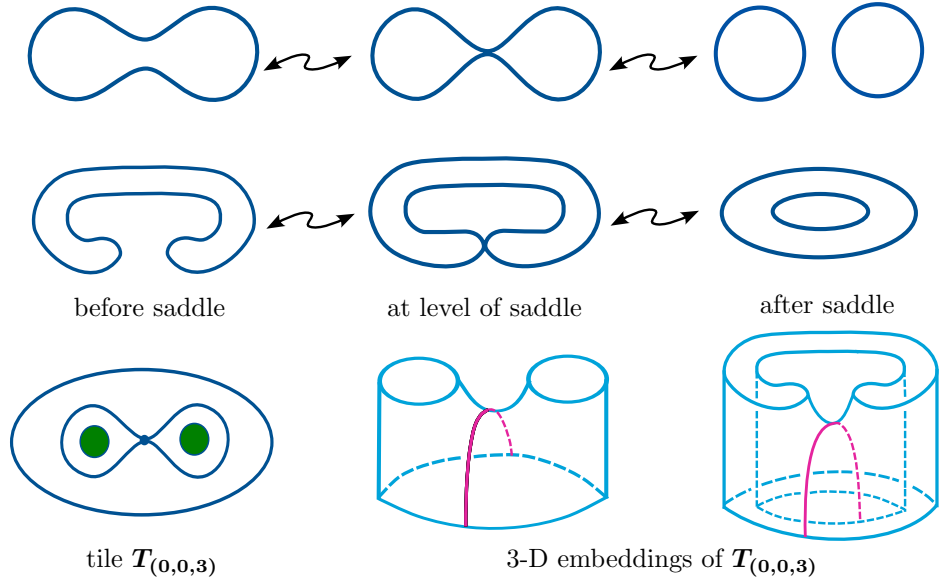

We classify the saddles corresponding to the tiles of type \bm{$T_{(2,2,1)}$} and \bm{$T_{(0,0,3)}$} into two categories: 

\begin{definition}
a) \textbf{Up saddle}: A saddle such that two curves pass through the saddle to join and become one curve, going in the direction of decreasing `$t$' i.e. moving down the plat height function. \\
In Figures ~\ref{beta} and ~\ref{gamma}, in the last row of both these figures, the 3-D embedding of the tile on the left both depict up saddles.

b) \textbf{Down saddle}: A saddle such that one curve passes through the saddle to split into two, going in the direction of decreasing `$t$'.\\
In Figures ~\ref{beta} and ~\ref{gamma}, in the last row of both these figures, the 3-D embedding of the tile on the right both depict down saddles.
\end{definition}

(iv), (v) \textbf{Points of extrema}: The last two singularities, described in Figure~\ref{deltaandtheta}, are points of minima (or maxima) on $\mbf{D}$. \bm{$T_{(0,0,1)}$} has the point of extrema lying in the interior $\mbf{D}$, away from the boundary, which basically means that a neighborhood of \bm{$T_{(0,0,1)}$} is just capping off a simple closed curve in the foliation of $\mbf{D}$. The tile \bm{$T_{(0,0,1)}$} is a disc with a circle boundary lying in the interior of $\mbf{D}$. We call a tile of type \bm{$T_{(0,0,1)}$}, a \textbf{min tile of type \bm{$T_{(0,0,1)}$}} (respectively, \textbf{max tile of type \bm{$T_{(0,0,1)}$}}), if the singularity in the tile is a local min (respectively local max).

\bm{$T_{(1,1,0)}$} has the point of extrema lying on the unknot itself, which means its neighborhood is one of the top or the bottom bridges. The tile \bm{$T_{(1,1,0)}$} is a disc with boundary as a straight line glued to a semicircle, refer Figure~\ref{deltaandtheta}. The semicircle lies on the boundary of \bm{$\mbf{D}$} (colored black in the Figure), while the straight line is in the interior of $\mbf{D}$, hence shown as blue in the Figure. Abusing notation, we also call a tile of type \bm{$T_{(1,1,0)}$}, a \textbf{min tile of type \bm{$T_{(1,1,0)}$}} (respectively, \textbf{max tile of type \bm{$T_{(1,1,0)}$}}), if the singularity in the tile is a local min (respectively local max).

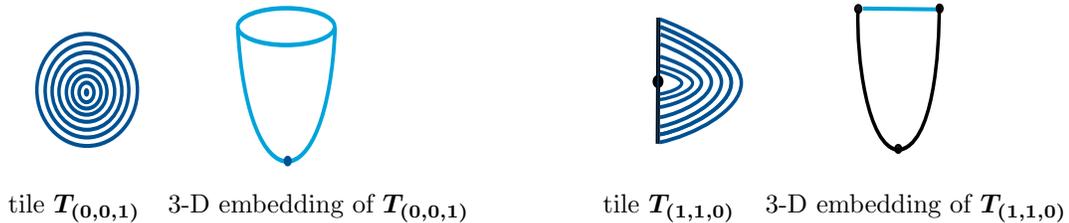
\begin{figure}[h!]
\labellist
\small\hair 2pt
\pinlabel {tile \bm{$T_{(0,0,1)}$}}  at 56 37
\pinlabel {3-D embedding of \bm{$T_{(0,0,1)}$}} at 238 37
\pinlabel {tile \bm{$T_{(1,1,0)}$}} at 499 37
\pinlabel {3-D embedding of \bm{$T_{(1,1,0)}$}} at 683 37
\endlabellist
\centering
\includegraphics[width=13cm, height=4cm]{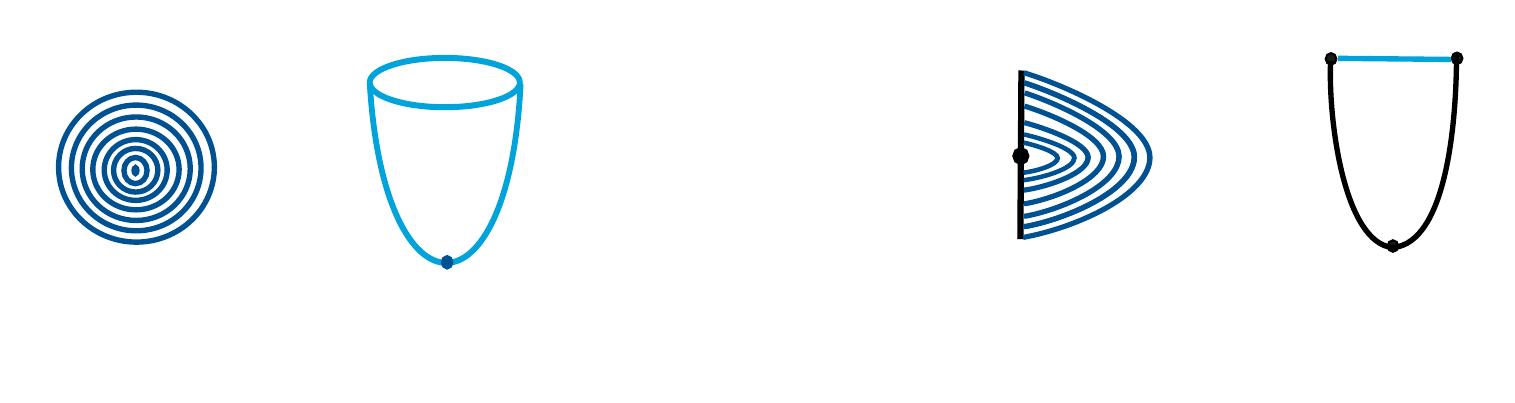}
\caption{Points of extrema - \bm{$T_{(0,0,1)}$} and \bm{$T_{(1,1,0)}$}}
\label{deltaandtheta}
\end{figure}

Note that the 3-D embeddings shown in the Figures in this section only depict the local geometric realisation of the spanning disc up to horizontal inversions, that is to say that saddles on the disc bounded by the unknot can also show up as horizontal inversions of the 3-D embeddings of the tiles depicted in the Figures in this section.

\subsection{Graph and complexity function} \label{graphcomplexity}


Given the tiling of a disc embedding $\mbf{D_i}$, we will now define a directed graph $G_{D_i}$ associated to it in the following way:

The set of vertices of the graph is the set of singularities (or tiles) in the foliation of the disc embedding $\mbf{D_i}$. There is an edge going from vertex $a$ to vertex $b$ if the tiles corresponding to the singularities are glued via an arc or a circle in the tiling of the disc and the singularity in the tile $a$ is at a larger $t$-value than the singularity in the tile $b$. Therefore, the valence of a tile $T_{(a,b,c)}$ is $b+c$. 

Notice that every edge in the graph corresponds to either an arc from boundary to boundary (of the disc) or a simple closed curve, both of which are separating curves on the disc. Now, if we had a cycle in the graph, that would give us a way to build a non-separating curve on the disc (because the boundary components of each tile $\mbf{\alpha}$ which do not lie on the unknot separate the disc into connected components and each vertex which shares an edge with $\mbf{\alpha}$ has to be in a different connected component of $D_i - \alpha$), which is impossible. Therefore, $G_{D_i}$ {\em is a tree for any embedding $D_i$}. 

We are now in a position to define the complexity function. Let $|T_{(4,4,0)}|$, $|T_{(0,0,1)}|$ be the number of singularities $T_{(4,4,0)}$, $T_{(0,0,1)}$ occurring on $D_i$. Then, define the complexity function on $D_i$ as follows: \\
$$ c(\mathbf{D_i}) = (|T_{(4,4,0)}|, |T_{(0,0,1)}|)$$ 
Lemma \ref{bla} in \S \ref{mainresult} tells us why it makes sense to define the complexity function this way.

\section{\centering The main result} \label{mainresult}

We start by proving the following Lemma: 

\begin{lemma}

Let $S_i$, $i = 1, 2, ..., n$ be $n$ simplicial complexes of degree 2 which are glued together to make a surface $S$. Let there be $|S_i|$ many $S_i$'s such that each such piece is glued in the same manner. Let $\chi(S_i)$ be the Euler characteristic of $S_i$, $e(S_i)$ be the number of edges of $S_i$ glued to an edge of $S_j$, for some j, where $ i \neq j$. Then, we have:

$$ \chi(S) = \sum_{i=1}^{n} |S_i| \chi(S_i) - \frac{1}{2} (\sum_{i=1}^{n} |S_i|e(S_i)) $$

\label{surfaceec}

\end{lemma}

\begin{proof}

For any simplicial complex $C$, let $n_i(C)$ be the number of $i$-cells in $C$. \\
Case(i): $S$ is obtained from $S_1$, $S_2$ by gluing $S_1$ with $S_2$ via circles, therefore $e(S_1) = e(S_2) =0$. Let $\tau$ be the gluing simple closed curve.\\
Now, let $S$ have a cell decomposition in which the circle $\tau$ is a subcomplex, and so $\tau$ consists of $k$ 0-dimensional cells and $k$ 1-dimensional cells for some $k \geq 1$. \\
Cutting $S$ along $\tau$ to get $S_1 \bigcup S_2$ and counting cells, we get:
\begin{equation}
n_0(S) = n_0(S_1) + n_0(S_2) - k
\end{equation}  
\begin{equation}
n_1(S) = n_1(S_1) + n_1(S_2) - k
\end{equation}
\begin{equation}
n_2(S) = n_2(S_1) + n_2(S_2)
\end{equation}
(1) - (2) + (3) gives us: 
$$ \chi(S) = \chi(S_1) + \chi(S_2) $$

Case(ii): $S$ is obtained from $S_1$, $S_2$ by gluing $S_1$ with $S_2$ via arcs, say, $\tau_i, i = 1,2, ..., p$ . As before, let $S$ have a cell decomposition in which the attaching arcs are subcomplexes. \\
Then $n_0(\tau_i) = k_i + 1$ and $n_1(\tau_i) = k$ for $i =1, 2, ..., p$. Again, cutting $S$ along these arcs and counting cells, we get: 

\begin{equation}
n_0(S) = n_0(S_1) + n_0(S_2) - ( k_1 + 1 + k_2 + 1 ... + k_p + 1) \\
= n_0(S_1) + n_0(S_2) - p - \sum_{i=1}^{p} k_i
\end{equation}
\begin{equation}
n_1(S) = n_1(S_1) + n_1(S_2) - ( k_1 + k_2 ... + k_p ) \\
= n_1(S_1) + n_0(S_2)  - \sum_{i=1}^{p} k_i
\end{equation}
\begin{equation}
n_2(S) = n_2(S_1) + n_2(S_2)
\end{equation}
(4) - (5) + (6) gives us: \\ 
\begin{align*}
\chi(S) &= \chi(S_1) + \chi(S_2) - p \\
          &= \chi(S_1) + \chi(S_2) - \frac{1}{2}(2p) \\
          &= |S_1|\chi(S_1) + |S_2|\chi(S_2) - \frac{1}{2}(|S_1| e(S_1) + |S_2| e(S_2)) \\
\end{align*}
         
Case(i) and case(ii) along with induction prove the lemma.
 
\end{proof}


\begin{lemma}

 $|T_{(0,0,1)}| = |T_{(2,2,1)}| + |T_{(0,0,3)}|$ \\

\label{bla}
\end{lemma}

\begin{proof} 

From Lemma~\ref{surfaceec}, we have the following equation: \\

\begin{align*}
 \chi(S) &= \sum_{i=1}^{n} |S_i| \chi(S_i) - \frac{1}{2} (\sum_{i=1}^{n} |S_i|e(S_i))  \\
          &= \sum_{i=1}^{n} |S_i|[\chi(S_i) - \frac{1}{2} e(S_i)] \\
\end{align*}

As stated above, we glue together neighborhoods of the 5 critical points to recover the disc bounded by the unknot and then applying this equation on that construction, we get: \\
\begin{align*}
1 &= |T_{(4,4,0)}|[\chi(T_{(4,4,0)})-\frac{1}{2}e(T_{(4,4,0)})] + |T_{(2,2,1)}|[\chi(T_{(2,2,1)})-\frac{1}{2}e(T_{(2,2,1)})] + |T_{(0,0,3)}|[\chi(T_{(0,0,3)})-\frac{1}{2}e(T_{(0,0,3)})] \\
   & +|T_{(0,0,1)}|[\chi(T_{(0,0,1)})-\frac{1}{2}e(T_{(0,0,1)})] + |T_{(1,1,0)}|[\chi(T_{(1,1,0)})-\frac{1}{2}e(T_{(1,1,0)})] \\
  &= |T_{(4,4,0)}|[1 - \frac{1}{2}(4)] + |T_{(2,2,1)}|[0-\frac{1}{2}(2)] + |T_{(0,0,3)}|[-1 - \frac{1}{2}(e(T_{(0,0,3)}))] + |T_{(0,0,1)}|[1-\frac{1}{2}(0)] + |T_{(1,1,0)}| [1-\frac{1}{2}] \\
  &= |T_{(4,4,0)}| [1-2] + |T_{(2,2,1)}| [-1] + |T_{(0,0,3)}|[-1 - \frac{e(T_{(0,0,3)})}{2}] + |T_{(0,0,1)}| + \frac{|T_{(1,1,0)}|}{2} 
  &= -|T_{(4,4,0)}| - |T_{(2,2,1)}| + |T_{(0,0,1)}| + \frac{|T_{(1,1,0)}|}{2} + |T_{(0,0,3)}|[-1 - \frac{e(T_{(0,0,3)})}{2}] \\
   &= \frac{|T_{(1,1,0)}|}{2} - |T_{(4,4,0)}| + |T_{(0,0,1)}| - |T_{(2,2,1)}| + |T_{(0,0,3)}|[-1-0]  \\
  &= \frac{|T_{(1,1,0)}|}{2} - |T_{(4,4,0)}| + |T_{(0,0,1)}| - |T_{(2,2,1)}| - |T_{(0,0,3)}|  
\end{align*}

Now, observe that every instance of the singularity \bm{$T_{(4,4,0)}$} increases the bridge index by one. We can see this by analysing the 3-D embedding of \bm{$T_{(4,4,0)}$} from Figure~\ref{alpha}. Notice that the saddle surface has 4 knot strings emanating from the level of the saddle. They extend in both directions, up and down, conceivably knotting in both directions. But, since we are working with a plat presentation of the unknot, these 4 strings will ultimately result in two bridges at the top and the bottom. Now, notice that each further instance of \bm{$T_{(4,4,0)}$} will add two more strings i.e. increase the bridge index of the diagram by 1. Thus we have: 

$|T_{(4,4,0)}| = n-1$ 

This gives us: $\frac{|T_{(1,1,0)}|}{2} - |T_{(4,4,0)}| = 1 $, where $\frac{|T_{(1,1,0)}|}{2}$ is the bridge index of the diagram, as explained in Section 3.

Therefore, we get: \\
$1 = 1 + |T_{(0,0,1)}| - |T_{(2,2,1)}| - |T_{(0,0,3)}| $ \\
$\Lra \mathbf{|T_{(0,0,1)}| = |T_{(2,2,1)}| + |T_{(0,0,3)}|}$
\end{proof}

We are now ready to prove Theorem 1:

\begin{proof}


Now, Lemma~\ref{bla} tells us that getting rid of one instance \bm{$T_{(0,0,1)}$} is equivalent to getting rid of either one instance \bm{$T_{(2,2,1)}$} or \bm{$T_{(0,0,3)}$}, thus making sense for us to define our complexity function the way we did at the end of \S \ref{graphcomplexity}. \\
Therefore, to reduce the unknot diagram to its standard presentation, we need only remove all instances of the singularities \bm{$T_{(4,4,0)}$} and \bm{$T_{(0,0,1)}$} and the crossings which result from braid isotopies and double coset moves. \\
We do that in the following manner: \\
We have, $c(D_i) = (|T_{(4,4,0)}|, |T_{(0,0,1)}|)$, with the dictionary order, is the complexity function associated with $D_i$, where $|T_{(4,4,0)}|$, $|T_{(0,0,1)}|$ are the number of occurrences of the singularities $T_{(4,4,0)}$ and $T_{(0,0,1)}$ on $D_i$. Further, $C_{t}^{i}$ is the set of level curves obtained from $D_i$ at level $t$. Now, our first step is to resolve any possible crossings using braid isotopies i.e. algebraic cancellations in the braid word corresponding to $K_i$. Then, we move on to our second step where we start removing singularities from the foliation of $D_i$. We will use the directed graph defined in the previous section to help keep track of the process of removing singularities from the disc foliation. \\

\begin{lemma}
For a given $K_i$ in the sequence of theorem 1, in the directed graph $G_{D_i}$, if there is a vertex $V$ which satisfies one of the following conditions: 

a) $V$ represents a min tile of type \bm{$T_{(0,0,1)}$} connected to a down saddle $N$ with neighbors $N_1$ and $N_2$ such that any level curve lying inside $N_1$, $N_2$ does not lie inside the simple closed curve corresponding to $V$ 

b) $V$ represents a max tile of type \bm{$T_{(0,0,1)}$} connected to an up saddle $U$ with neighbors $U_1$ and $U_2$ such that the level curves corresponding to $U_1$, $U_2$ do not lie inside the simple closed curve corresponding to $V$ 

c) $V$ represents a tile of type \bm{$T_{(4,4,0)}$} which is connected to at least one min tile of type \bm{$T_{(1,1,0)}$}, say $M_1$, and at least one max tile of type \bm{$T_{(1,1,0)}$}, say $M_2$, such that when traversing $K_i$, if we write the sequence of tiles we hit, we get $M_1 V M_2$ or $M_2 V M_1$ as a subsequence 

then, we can remove the vertex $V$ using the flip or the pocket move in cases a) and b) and a generalised destabilization in case c), so that we get to $K_{i+1}$ in the sequence given in theorem 1, where $c(D_{i+1}) < c(D_i)$.

\label{lemma5}

\end{lemma}

\begin{proof}


\begin{figure}[h!]
\labellist
\small\hair 2pt
\pinlabel {$\alpha$}  at 193 90
\pinlabel {$V$} at 121 46
\pinlabel {$V_1$} at 230 46
\pinlabel {$\mbf{B}$} at 110 236
\pinlabel {$S_{V_1}$} at 177 170
\pinlabel {cross section of $N(\alpha)$} at 400 129
\pinlabel {$N(\alpha)$} at 265 180
\endlabellist
\centering
\includegraphics[width=11cm, height=8cm]{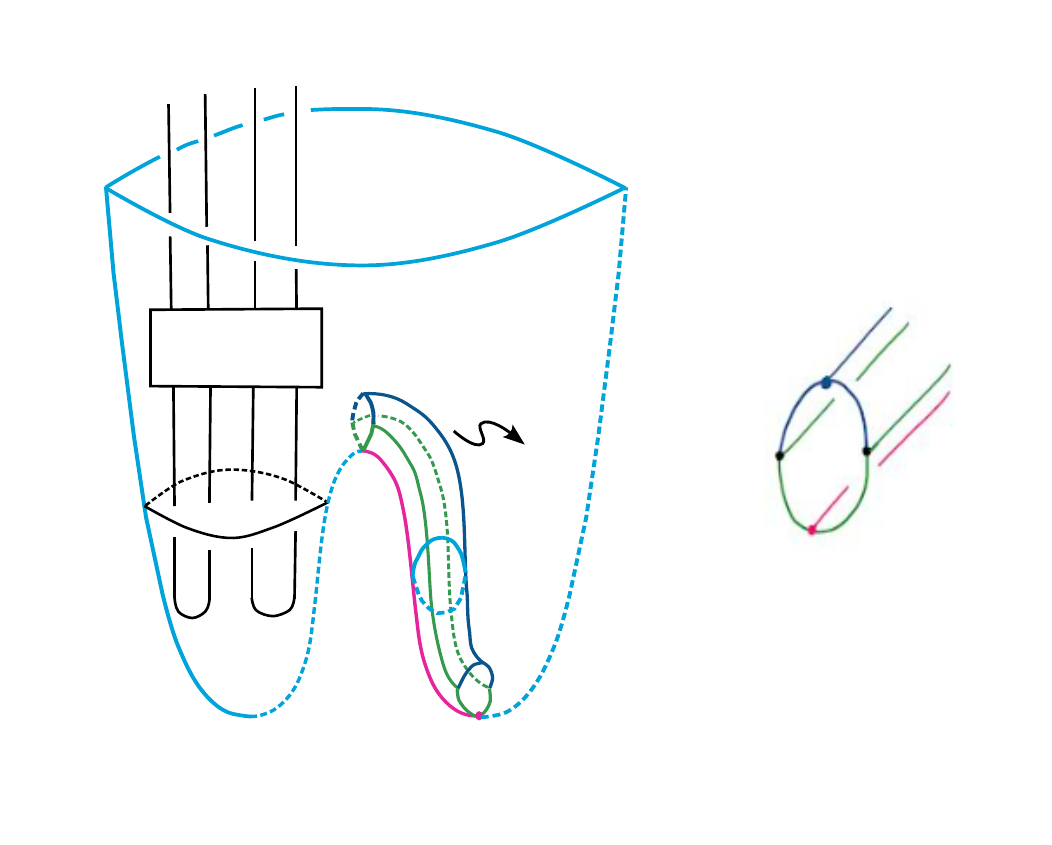}
\caption{Path isotopy to remove saddles}
\label{path_isotopy}
\end{figure}
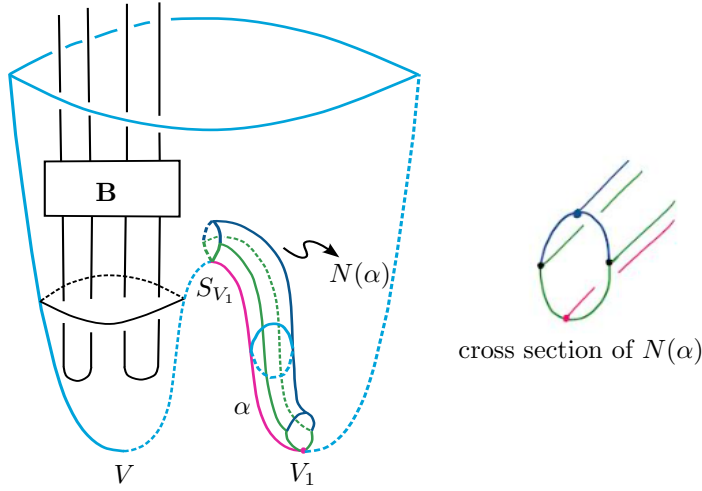

Consider a vertex $V$ satisfying condition a). Then, we look for a closest min tile to $V$, in the graph $G_{D_i}$, say $V_1$. Note that such a tile always exists because we have $|T_{(4,4,0)}| + 1$ many min tiles of type $T_{(1,1,0)}$, corresponding to the bottom bridges. Consider the saddle $S_{V_1}$ that the singularity $V_1$ is induced by. Now, define a path $\alpha$ (the pink curve in Figure \ref{path_isotopy}), from $S_{V_1}$ to $V_1$, on the disc, such that $\alpha$ is always transverse to the foliation, other than at points $\alpha(0)$, $\alpha(1)$. Consider a thickened up regular neighborhood of $\alpha$, say $N(\alpha)$ which stays completely on one side of the disc, as shown in Figure \ref{path_isotopy}. Notice that $N(\alpha)$ is a 3-ball, positioned so that the boundary sphere is split up into the following four components: two copies of $\alpha \times I$ (the green part in the figure which lives on the surface, and the blue part which is completely on one side of the surface) glued along their respective $\alpha \times \{0\}$ and $\alpha \times \{1\}$ components. This gives us an annulus. Capping off this annulus with two discs gives us the two sphere which is the boundary of $N(\alpha)$. \\
For any singularity $S$ on $D_i$, let $t_{S}$ denote the level at which the singularity occurs. Now, we inspect $C^{i}_{t}$, for some $t \in (t_N, t_N+\epsilon)$. The arcs inside the curve corresponding to $V$, we remove via the following isotopy: we move the bridges corresponding to these arcs along the path $\alpha$, always staying in side the $3$-ball $N(\alpha)$, starting from the point $V$ on the boundary of $N(\alpha)$ and ending at the point $V_1$, also on the boundary of $N(\alpha)$. Notice that this corresponds to doing a { \em pocket move}. Once we have emptied out the simple closed curve corresponding to $V$, we can surger off the subdisc which is the tile represented by $V$, thus simultaneously cancelling the saddle $N$, while removing the vertex $V$. \\
We can do a similar isotopy for a vertex $V$ satisfying condition b), with the only difference being that we now need to look for the next closest max tile to $V$. \\
For the case of condition c), the tiles $M_1$, $M_2$, alongwith the part of tile $V$ sandwiched between them bound a subdisc devoid of any singularities. Since the subdisc does not have any singularities, it depicts a trivial loop, which can be removed via a standard destabilization after doing some double coset moves which align the top and bottom bridge at the same position. This gets rid of the tiles $M_1$, $M_2$ and $V$. \\
In all the three cases above, we get a new plat $K_{i+1}$ such that $C(D_{i+1}) < c(D_i)$.
\end{proof}

\begin{lemma}
A vertex satisfying condition a) or b) in Lemma \ref{lemma5} always exists, if there are any tiles of type \bm{$T_{(0,0,1)}$} in the foliation of the disc embedding $D_i$. If there is a tile of type \bm{$T_{(4,4,0)}$} in the foliation, but no tile of type \bm{$T_{(0,0,1)}$}, then there exists a vertex satisfying condition c) in Lemma \ref{lemma5}.

\label{lemma6}

\end{lemma}

\begin{proof}


\begin{figure}[h!]
\labellist
\hair 6pt
\pinlabel {connecting scc}  at 158 156
\pinlabel {up saddle} at 453 175
\pinlabel {down saddle} at 55 225
\pinlabel {$s_1$} at 205 206
\pinlabel {$s_2$} at 311 170
\pinlabel {$b$} at 155 178
\pinlabel {$c$} at 292 187
\pinlabel {$d$} at 385 200
\endlabellist
\centering
\includegraphics[width=11cm, height=8cm]{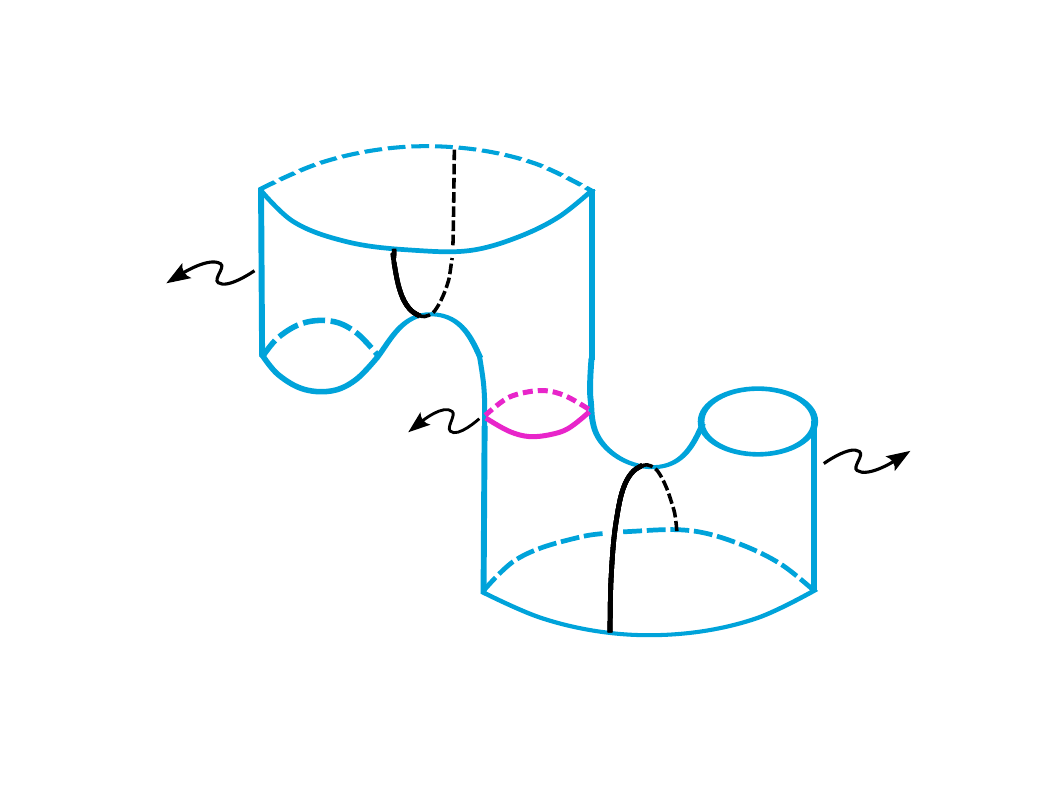}
\caption{A down saddle connected to an up saddle via a simple closed curve}
\label{updownsaddle}
\end{figure}
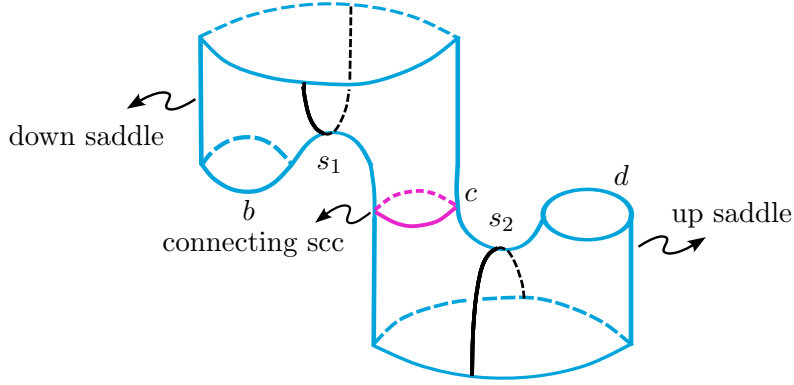

We argue this by contradiction. If there were no such vertices as described in the previous lemma, then no down saddle is connected to a min tile and no up saddle is connected to a max tile. \\
Then, near any down saddle $s_1$, the disc locally looks like as depicted in Figure~\ref{updownsaddle}, with the variation that $s_1$ could also have been connected to another down saddle, as opposed to the up saddle shown in the figure. Now, since the saddle $s_1$ is not connected to a min tile, neither of the two curves $b$, $c$ are capped off by discs. Therefore, $s_1$ is connected to two more saddles, say $s_2$ and $s_3$. Note that $s_2$ and $s_3$ cannot be connected to each other because that would mean the surface has genus, which is a contradiction. Then, if both $s_2$ and $s_3$ are down saddles, by the same argument, as above, we get two more down saddles each for $s_2$ and $s_3$, because they are not connected to any min tiles either. If this process continues indefinitely, we have infinitely many tiles, which can not happen. \\
Therefore, without loss of generality, let $s_2$ be an up saddle. Now the local picture is exactly as shown in Figure~\ref{updownsaddle}. Further, since no up saddle is connected to a max tile, $d$ cannot be capped off by a disc, therefore $d$ is further connected to another saddle and, as before, this process continues indefinitely unless there is an up saddle connected to a max tile or down saddle connected to a min tile. \\
If we locate a min tile $M$ connected to a down saddle $s_1$, such that the scc corresponding to $M$ contains an arc coming from another tile connected to $s_1$, then we use the flip or the microflip move to remove that arc (this might potentially get rid of multiple instances of $T_{(0,0,1)}$. Doing this, we now have a vertex satisfying condition a) of the previous lemma. The argument is exactly the same for a max tile. \\
Now, let $|T_{(0,0,1)}| = 0$ and $|T_{(4,4,0)}| > 0$. Then, by Lemma ~\ref{bla}, $|T_{(0,0,3)}| = 0$ and $|T_{(2,2,1)}| = 0 $. Therefore, any tile of type $T_{(4,4,0)}$ has to be connected to 4 tiles of type $T_{(0,0,1)}$, two of which will be min tiles and two max tiles (refer the bottom right picture in Figure ~\ref{alpha}). Therefore, this $T_{(4,4,0)}$ tile corresponds to a vertex $V$ in the graph satisfying condition c) of Lemma ~\ref{lemma5}.
\end{proof}

In light of Lemmas \ref{lemma5} and \ref{lemma6}, we now have a blueprint which tells us how to create the sequence in Theorem \ref{thm1}, monotonically decreasing the complexity function on the $K_i$'s. We terminate the sequence when we reach a $K_l$ which does not have any singularities on the disc it bounds, other than a top bridge tile and a bottom bridge tile. Then, any remaining crossings come from twisting the bridges or doing Reidemeister II moves with consecutive strands coming from the same bridge, which can be undone using the double coset moves or braid isotopies, leading us to the standard presentation of the unknot. This completes the proof of Theorem \ref{thm1}.
\end{proof}

We now prove Corollary \ref{case_unlink}.

\begin{proof}

Since $\mbf{L}$ is a presentation of $\mbf{U_k}$, it bounds $k$ pairwise disjoint discs. As described previously, each of these discs is foliated by level curves. \\
Then, we start by analysing the foliation on each of these discs. \\
Starting with the disc corresponding to the first bridge, say, $D_1$, we remove all singularities from it, by the algorithm described above. we are able to do this since the collection of discs is disjoint. \\
We now analyse the system of level curves corresponding to the discs. If, at each level, barring $0$ and $1$, the system of level curves is standard i.e. $k$ straight line segments, such that none is in front of any other, there are no other crossings and the diagram is the standard $k$-bridge, 0-crossing diagram of the unlink. \\
 If not, then we get rid of the remaining crossings using braid isotopies and double coset moves. We are able to do so since none of the discs have any singularities of type $T_{(4,4,0)}$, $T_{(2,2,1)}$ or $T_{(0,0,1)}$. \\ 
 We repeat this process till we reach the standard diagram, which finishes our proof.
\end{proof}

\section{\centering Situation in $S^3$} \label{S^3}

We think of $S^3$ as $\mathbf{S}(S^2) = S^2 \times [0,1]/ S^2 \times \{0\} , S^2 \times \{1\} $, the double suspension of $S^2$. Then the height function for a plat presentation will have the $S^2 \times \{ t \} , \ 0 < t < 1 $ as the level surfaces.  With this structure one observes that the flip move is a braid isotopy in $S^3$. Specifically, in Figure ~\ref{stab_seq} the pink strand is always transverse to the level spheres.  Thus, we reinterpret the stabilisation sequence in Figure ~\ref{stab_seq} as a sequence of braid isotopies for $B_n(S^2)$. This is explained very lucidly and in detail in Chapter 10 of \cite{murasugi_kurpita}. This gives us the following Corollary: \\
\begin{corollary} 

\label{one_dblcosetcls}

In $B_n(S^2)$, there is only one double coset class corresponding to the unknot, the obvious representative of which is the $n$-stabilised trivial plat, $U_n$.

\end{corollary}
\begin{proof}

In light of the fact that the flip move is a braid isotopy in $B_n(S^2)$, and pocket moves are a combination of double coset moves and braid isotopies, Theorem \ref{mainresult} tells us that we can go from any $n$-bridge plat representative of the unknot to the trivial plat using just braid isotopies, double coset moves and destabilisations. \\
Now, if $K_1$ is an $n$-bridge plat representative of the unknot, then we can go from $K_1$ to the trivial plat using double coset moves and destabilizations (and of course, braid isotopies, but they do not change the braid corresponding to a plat). Then, from Lemma 11 in \cite{birman_1976}, we can perform all the double coset moves before all the destabilizations, which means that we can always first reduce $K_1$ to $U_n$, the $n$-stabilised trivial plat, before destabilising $n$ times.
\end{proof}

We would like to point out here that Jean Pierre Otal shows in \cite{otal} that in $S^3$, any $n$- bridge representative of the unknot is bridge isotopic to $U_n$, and Corollary \ref{one_dblcosetcls} implies his result.

\section*{The Goeritz unknot}  \label{goeritz}


\begin{figure}[h]

\centering

\includegraphics[width=6cm, height=6cm]{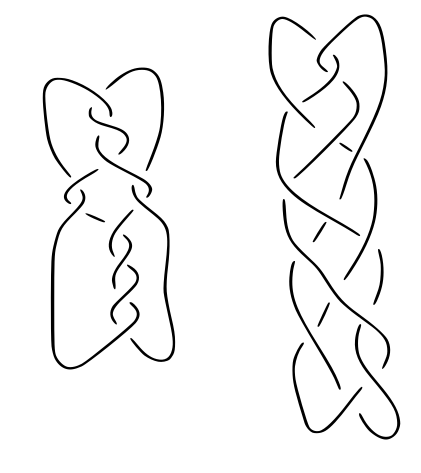}
\caption{The Goeritz unknot and a variation}
\label{goeritzvmine}

\end{figure}

\begin{figure}[ht!]
\centering
\includegraphics[width=17cm, height=12.5cm]{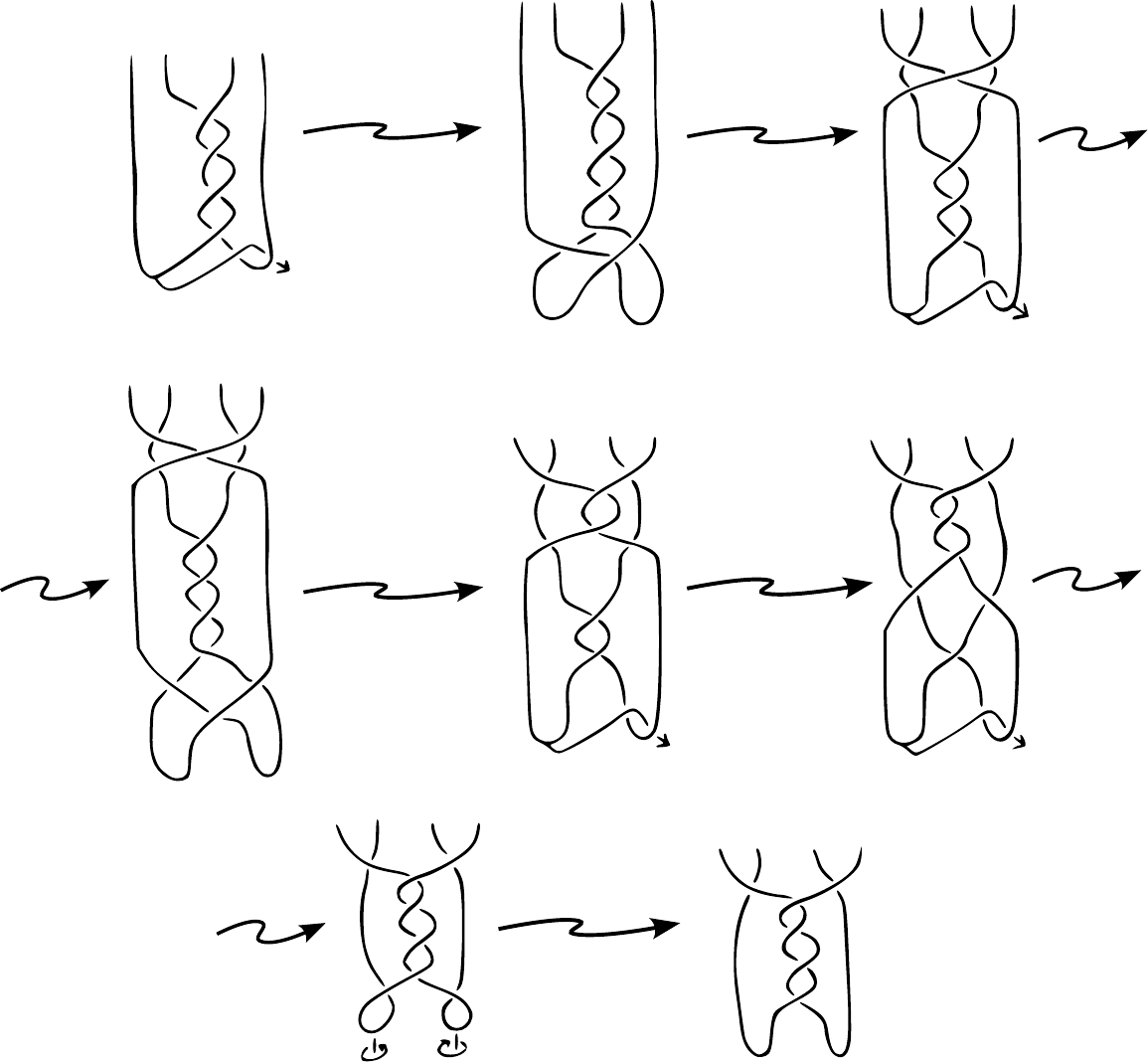}
\caption{Isotopy Sequence A}
\label{iso_seq_A}
\end{figure}

\begin{figure}[ht!]
\centering
\labellist
\pinlabel {I} at 111 234
\pinlabel {II} at 296 234
\pinlabel {III} at 476 234
\pinlabel {IV} at 235 29
\pinlabel {V} at 409 29
\pinlabel {Isotopy Seq. A} at 373 376
\pinlabel {$R$ $II$} at 190 378
\pinlabel {$R$ $II$} at 536 378
\pinlabel {Destabilization} at 322 160
\endlabellist
\centering
\includegraphics[width=15cm, height=10cm]{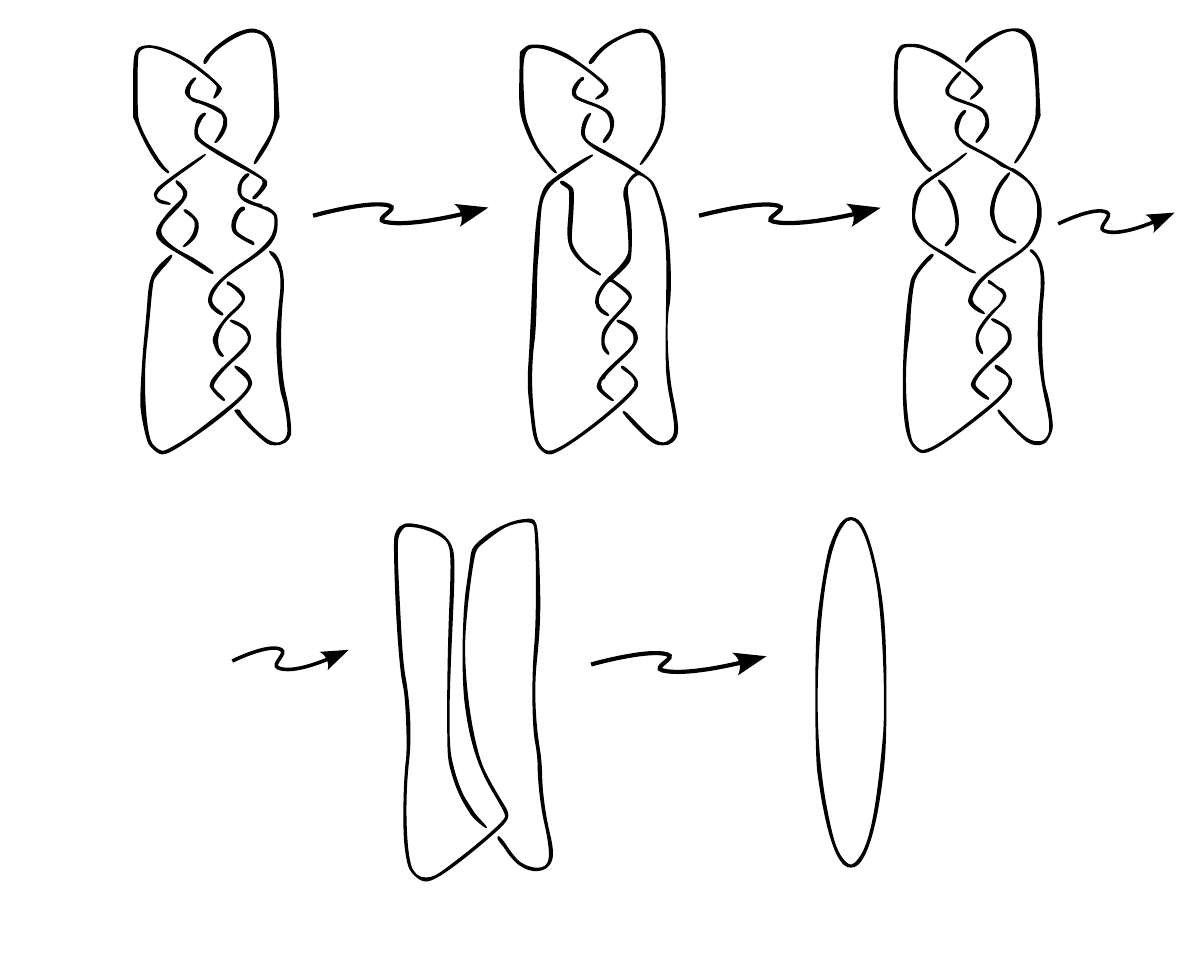}
\caption{Unknotting Goeritz unknot using double coset moves}
\label{unknottinggoeritz}
\end{figure}

\begin{figure}[ht!]
\centering
\labellist
\small\hair 2pt
\pinlabel {Flip} at 130 406
\pinlabel {$R$ $II$} at 250 406  
\pinlabel {$R$ $II$} at 363 406
\pinlabel {$R$ $II$} at 463 406
\pinlabel {Destabilization} at 326 170
\endlabellist
\centering
\includegraphics[width=15cm, height=12.5cm]{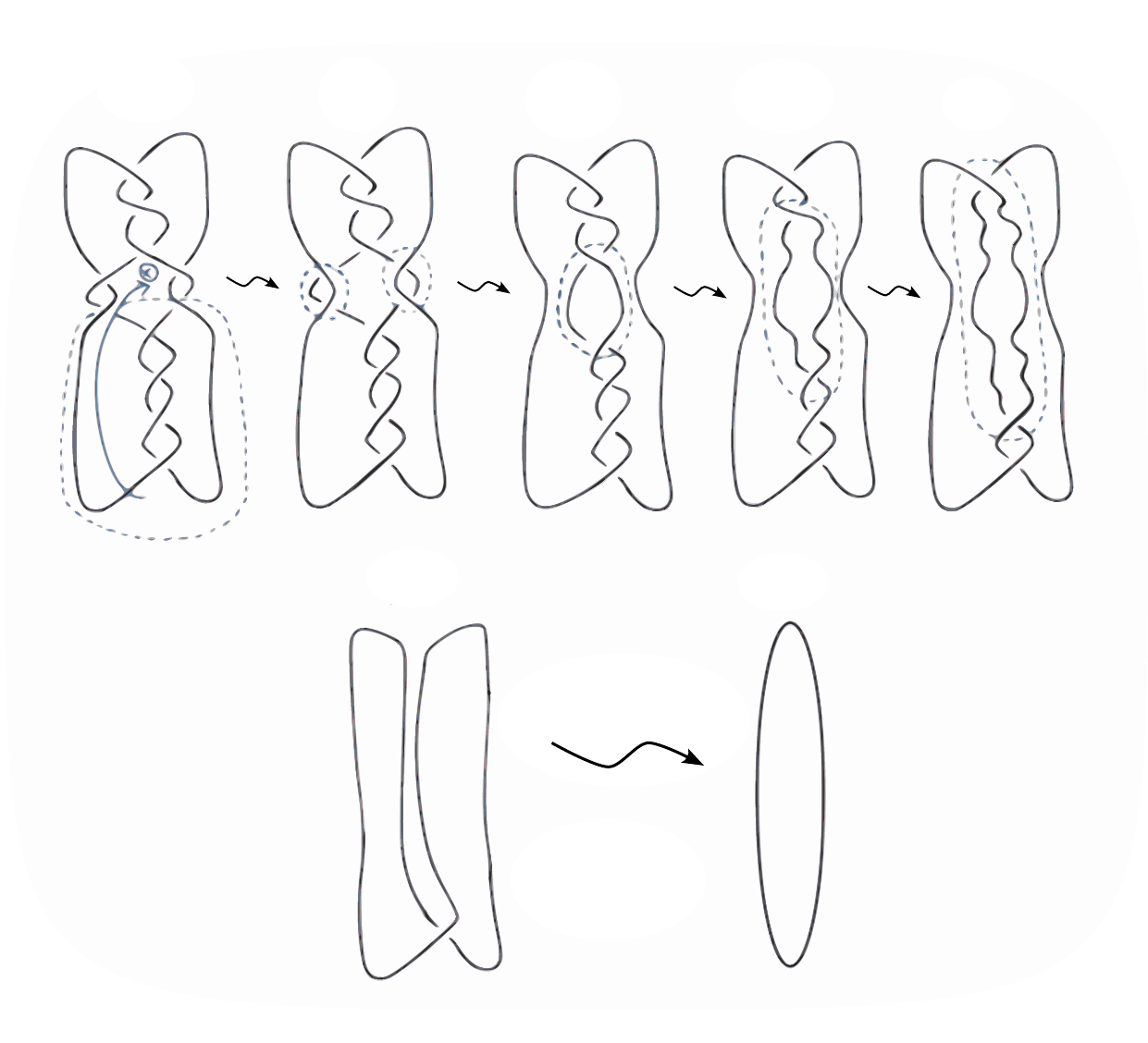}
\caption{Unknotting the Goeritz unknot using the flip move}
\label{goeritzunknotting}
\end{figure}

Figure ~\ref{goeritzvmine} shows two 2-bridge diagrams of the unknot which (prior to the discovery of the pocket and flip moves) require stabilization (in the geometric sense) in order to be simplified to the 0-crossing unknot. The one on the left is due to Goeritz (first ever appearance in \cite{goeritz_34}), the second one due to the author. The knot software Regina \cite{regina} can be readily used to check the fact that geometric stabilization (i.e. adding a crossing) is necessary. This can be done by restricting the set of moves to a subset of the Reidemeister graph (\cite{barbensi_2020}) where the number of crossings does not exceed that of the initial diagram. 

We show how to untangle the Goertiz unknot both with and without the flip move. First consider Isotopy sequence A, applied to the bottom part of the Goeritz unknot, shown in Figure ~\ref{iso_seq_A}. Applying this sequence to the Goeritz unknot, we get the first plat in Figure ~\ref{unknottinggoeritz}, which is then simplified to the trivial diagram as shown in the figure. Contrast this with Figure~\ref{goeritzunknotting}, which shows untangling of the Goeritz unknot using the flip move, without ever exceeding the number of crossings in the initial diagram. Notice that untangling of the Georitz using the flip move is much faster than without it. 

Untangling the second unknot in Figure ~\ref{goeritzvmine}, {\em without stabilization}, utilises both the flip move and certain double coset moves (in addition to braid isotopies and destabilization).






\bibliography{bib}{}
\bibliographystyle{alpha}
\addcontentsline{toc}{section}{References}

\end{document}